\documentclass{amsart}
\usepackage{amsmath,amssymb}
\usepackage{verbatim}
\usepackage[shortlabels]{enumitem}
\newtheorem{theorem}{Theorem}[section]
\newtheorem{lemma}[theorem]{Lemma}
\newtheorem{corollary}[theorem]{Corollary}

\theoremstyle{definition}
\newtheorem{definition}[theorem]{Definition}
\newtheorem{remark}[theorem]{Remark}

\newcommand{\column}[2]{\begin{pmatrix}#1\cr #2\end{pmatrix}}
\newcommand{\A}{\mathcal A}
\DeclareMathOperator{\Gal}{Gal}
\newcommand{\F}{\mathbb F}

\newcommand{\Q}{\mathbb Q}
\newcommand{\Z}{\mathbb Z}
\newcommand{\gl}{{\rm GL}}

\newcommand{\GL}{{\rm GL}}
\DeclareMathOperator{\diag}{diag}
\newcommand{\SL}{{\rm SL}}
\newcommand{\T}{\mathcal T}

\newcommand{\lcm}{{\rm LCM}}

\newcommand{\transpose}[1]{{}^t#1}
\newcommand{\transfer}{{\rm tr}}

\newcommand{\corestriction}{i}

\newcommand{\OK}{\mathfrak O}
\newcommand{\Ol}{\mathfrak O_\ell}
\newcommand{\Ow}{\mathfrak O_w}
\newcommand{\homothety}{\mathcal H}
\newcommand{\basepointy}{{\bar\omega}}
\newcommand{\basepoint}{{\tilde\omega}}

\DeclareMathOperator{\Frob}{Frob}
\DeclareMathOperator{\tr}{Tr}
\DeclareMathOperator{\ind}{Ind}

\newcommand{\KSQ}{K_{S,q}}
\newcommand{\XSQ}{X_{S,q}}
\newcommand{\xS}{x}

\newcommand{\MSQ}{M_{S,q}}

\newcommand{\stab}[1]{\Gamma_{#1}}
\newcommand{\stabhat}[1]{\hat\Gamma_{#1}}

\newcommand{\zreps}{\mathcal B}
\newcommand{\ZSQ}{Z_{S,q}}
\newcommand{\zxs}{z_{\xS}}
\DeclareMathOperator{\disc}{disc}

\newcommand{\wpart}{I_\alpha(w)}
\DeclareMathOperator{\gal}{Gal}
\DeclareMathOperator{\frob}{Frob}
\DeclareMathOperator{\Ad}{Ad}
\newcommand{\zorbitreps}{\mathcal C}

\title{Even Galois representations and the cohomology of $\GL(2,\Z)$}
\author{Avner Ash}
\address{Boston College, Chestnut Hill, MA  02467}
\email{Avner.Ash@bc.edu}
\author{Darrin Doud}
\address{Brigham Young University, Provo, UT  84602}
\email{doud@math.byu.edu}
\date{February 23, 2017}

\begin{document}
\begin{abstract}
Let $\rho$ be a  two-dimensional even Galois representation which is induced from a character $\chi$ of odd order of the absolute Galois group of a real quadratic  field.  After imposing some additional conditions on $\chi$, we attach $\rho$ 
 to a Hecke eigenclass in the cohomology of $\gl(2,\Z)$ with coefficients in a certain infinite-dimensional vector space over a field of characteristic not equal to 2.
\end{abstract}
\maketitle

\section{Introduction}

In this paper a Galois representation will be a continuous representation $\rho:G_\Q\to\gl(n,\F)$ where $\F$ is either a topological field of characteristic 0 or a finite field.  When the characteristic of $\F$ is not two, we say that $\rho$ is odd if the image of complex conjugation is conjugate to a diagonal matrix with alternating $1$'s and $-1$'s on the diagonal.  If $\F$ has characteristic two, every Galois representation is considered to be odd.  When $n=2$, $\rho$ is odd, and $\F$ is a finite field, Serre's conjecture \cite{Serre-1987} (now a theorem of Khare and Wintenberger \cite{KW1,KW2}) states that $\rho$ is attached to a modular form that is an eigenform of the Hecke operators.  This means that the characteristic polynomial of the image of a Frobenius element at an unramified prime $\ell$ under $\rho$ equals a certain polynomial created from the eigenvalues of the Hecke operators at $\ell$.  Other papers \cite{ADP,AS,Herzig} conjecture a similar attachment for $n\geq 2$, with modular forms replaced by elements of arithmetic cohomology groups.  Work of Scholze \cite{Scholze} proves that any eigenclass of the Hecke operators in the cohomology of a congruence subgroup of $\SL(n,\Z)$ with coefficients in a finite-dimensional admissible module $M$ over a field $\F$ has an attached Galois representation. For a field $\F$ of characteristic 0, this theorem was already proven in \cite{HLTT} by Harris, Lan, Taylor and Thorne. Caraiani and Le Hung \cite{Caraiani} showed that the representation guaranteed by Scholze's theorem must be odd.  (``Admissible" means that  if $\F$ has characteristic 0 then $M$ is an algebraic representation, and if $\F$ has positive characteristic, then the matrices used to define the Hecke operators act on $M$ via reduction modulo some fixed integer.)

In this paper, we attach certain even Galois representations to eigenclasses in arithmetic cohomology groups.  The details of our main result may be seen in Theorem~\ref{T:main}, at the end of the paper.  Following \cite{Caraiani} we know that we will need to use a non-admissible, infinite dimensional coefficient module for the cohomology.  We also have to be careful with the exact definition of ``attachment", which we now explain.

Let $f$ be a  modular form of weight $k\geq 0$ on the upper half plane, with level $\Gamma_1(N)$ and nebentype $\theta$, and suppose that $f$ is an eigenform for the Hecke operators $T_\ell$ and $T_{\ell,\ell}$ for all $\ell\nmid N$.  Denote the eigenvalue of $T_\ell$ by $a_\ell$, and the eigenvalue of $T_{\ell,\ell}$ by $A_\ell$.  When $k\geq 2$, and $f$ is holomorphic, there is a Galois representation $\rho$ such that for all $\ell\nmid N$,
\[\det(I-\rho(\frob_\ell)X)=1-a_\ell X+\ell A_\ell X^2,\]
where (in this case), $A_\ell$ is easily seen to be equal to $\ell^{k-2}\theta(\ell)$.

The cases when $k=0$ or $1$ are different,  because then $f$ is not cohomological.  If $k=1$ and $f$ is holomorphic, or if $k=0$ and $f$ is a Maass form where the eigenvalue of the Laplacian is $1/4$, there is an attached Galois representation (this is only conjectural in the Maass form case) with finite image.  In both cases, the motivic weight of $f$ is 0, and the characteristic polynomial of 
$\Frob_\ell$ equals
$$
1-a_\ell X + \theta(\ell) X^2.
$$
These forms of the Hecke polynomials depend on the usual normalization of the Hecke operators.

In this paper, we will deal with the cohomology of $\gl(2,\Z)$ with a nonadmissible, infinite dimensional coefficient module and an even Galois representation.  Using the same  normalization of the Hecke operators as in the finite dimensional coefficient setting, we define attachment of the Galois representation as follows (compare \cite[Def. 2.1]{ADP}).
Note that this definition of ``attachment" is for the purposes of this paper only, although we may guess that it will be the correct definition to use for all even two-dimensional Galois representations.

\begin{definition}
Let $V$ be a Hecke module over the field $\F$, and let $v\in V$ be an eigenvector for the Hecke operators $T_\ell$ and $T_{\ell,\ell}$ for almost all primes.  Let $a_\ell$ be the eigenvalue of $T_\ell$ acting on $v$, and $A_\ell$ the eigenvalue of $T_{\ell,\ell}$ acting on $v$. Let $\rho:G_\Q\to\gl(2,\F)$ be a Galois representation.  We say that $\rho$ is attached to $v$ if, for almost all $\ell$,
\[\det(I-\rho(\frob_\ell)X)=1-a_\ell X+A_\ell X^2.\]
\end{definition}

Our main theorem (Theorem~\ref{T:main}) then takes the following form.

Let $K$ be a real quadratic field of discriminant $d$, let $\F$ be a field with characteristic not equal to 2, 
and let $\chi:G_K\to\F^\times$ be a character with finite image, satisfying certain conditions (described in Theorem~\ref{T:main}).  Then $\rho:G_\Q\to\gl(2,\F)$ given by $\rho=\ind_{G_K}^{G_\Q}\chi$ is an even Galois representation, and is attached to a Hecke eigenclass in $H^1(\gl(2,\Z),\MSQ^*)$, where $q$ is a character related to $K$ (see Definition~\ref{D:q*}), $\MSQ$ is defined in Definition~\ref{D:MSQ}, and the asterisk denotes $\F$-dual. 

The coefficient module $\MSQ$ that we use is naturally defined in terms of the field $K$.
It is a non-algebraic infinite-dimensional module somewhat related to the kind that we we used in \cite{AD1} to study reducible cases of the Serre-type conjecture for $GL(3)/\Q$.  

Any Maass eigenform with eigenvalue $1/4$ is conjectured to have a Galois representation attached.  Also, the Galois representations we work with in this paper are known to be attached to Maass forms.  Our innovation is to prove attachment to something cohomological.  Besides the intrinsic interest of this, we hope to be able to use our main theorem, combined with techniques similar to those of \cite{AD1}, to prove a Serre type conjecture for the sum of $\rho$ and a
character such that the three-dimensional representation as a whole is odd, in the context of $GL(3)/\Q$.

The idea of our proof is the following.   We view $K$ as a two-dimensional $\Q$-vector space.  We construct a $GL(2,\Q)$
module $M$  consisting of formal sums of homothety classes of $\Z$-lattices in $K$, where the homotheties are given as multiplication by the elements of a carefully chosen 
subgroup $\KSQ$ of $K^\times$.  
We use homothety classes, rather than the lattices themselves, so that the stabilizer of a homothety class in $\gl(2,\Z)$ will be an infinite cyclic group generated by the image $g$ of a unit in the ring of integers of $K$ under an embedding of $K$ into $GL(2)$ as a non-split torus.  This matrix $g$ also stabilizes a closed geodesic in the quotient of the upper half plane modulo $\gl(2,\Z)$.  Our initial idea was to work with the fundamental classes of these closed geodesics, but of course we don't want to view them in the homology of $\gl(2,\Z)$ with admissible coefficients for the reasons stated above.  Instead we use the more algebraic approach of this paper.

 We focus on the submodule $\MSQ$ of $M$  which consists of formal sums that have finite support modulo the center and which have central character $q$.  As we just said, the stabilizer of a homothety class of lattices is an infinite cyclic group.  This allows us to use Shapiro's lemma to write the homology $H_1$ of $\gl(2,\Z)$ with coefficients in 
 $\MSQ$ in terms of the $H_1$ of these cyclic stabilizers, which is an algebraic version of the fundamental classes of the corresponding closed geodesics.  We then have to understand how the Hecke operators act.

We use the method of partial Hecke operators described in \cite{AD1} to get a tractable formula for the action of a Hecke operator on  $H_1(\gl(2,\Z),\MSQ)$.  Now a class in that homology group has finite support (modulo the center) on chains, and the Hecke operators always expand the support.  So there will not be any Hecke eigenvectors in the homology group $H_1(\gl(2,\Z),\MSQ)$.  We must seek for Hecke eigenvectors in the dual space $H^1(\gl(2,\Z),\MSQ^*)$.  We interpret elements of the dual space as functions on the space of lattices in $K$.  In order for us to construct such functions in a way that makes it possible to compute the Hecke operators at 
$\ell$, we use the Bruhat-Tits graph $T_\ell$ for $\gl(2,\Q_\ell)$ or a double cover $\T^2_\ell$ of $\T_\ell$, depending on whether 
 $\ell$ is  split or inert in $K$.  We then
 relate the Hecke operators at $\ell$ to a Laplacian on $\T_\ell$ (or $\T^2_\ell$) and to the action of the center.  This allows us to construct lattice functions that have the desired Hecke eigenvalues.  
 
 These functions are restricted infinite products over the rational primes of the local functions we construct on the graphs.  Lattices which are fractional ideals in $K$ play a special role in the study of $\MSQ$ and we call them ``idealistic" lattices. 
 The construction of the local functions depends on the crucial distinction between idealistic and non-idealistic lattices. To define a cohomology class, the infinite product has to satisfy a certain global invariance property (proved in Section~\ref{S:Invariance}), which is guaranteed  by the fact that $\chi$ can be viewed as a global character on ideals.

In case $\F$ has characteristic 2, where the distinction between odd and even Galois representations breaks down, a simplified version of our construction works to attach $\rho$ to a Hecke eigenclass in the cohomology of $\gl(2,\Z)$ with coefficients in the analog of $\MSQ$ (where now $q$ would be the identity central character.)  In this case, we can work directly with the usual Bruhat-Tits graph for both split and inert  primes. We do not cover this in this paper because when the characteristic of $\F$ equals 2, $\rho$ is deemed to be odd and is known to be attached to a homology class with admissible coefficients, by the work of Khare and Winteberger \cite{KW1,KW2} referred to above.

We thank Dick Gross, David Hansen, Richard Taylor, and especially Kevin Buzzard for helpful comments and answers to questions that arose during the course of this research.

\section{Lattices and Homotheties in $K$}\label{S:Lattices}

Fix a real quadratic field $K$, its ring of integers $\OK$ and an element $\omega\in\OK$ such that $\OK=\Z[\omega]$.  Let $d$ be the discriminant of $K/\Q$. Let $\epsilon$ be a fundamental unit, i.e.{} a unit whose image modulo $\pm1$ generates $\OK^\times/\{\pm1\}$.

Consider $K$ as a two-dimensional vector space over $\Q$.  By a lattice in $K$, we will mean a free $\Z$-module of rank $2$ contained in $K$.  Such a module has as a $\Z$-basis two $\Q$-linearly independent elements.

Let $Y$ be the set of all column vectors $\transpose{(a,b)}\in K^2$ with $b\neq 0$ and $a/b\notin \Q$. If we let 
$\basepointy=\transpose{(\omega,1)}\in Y$, then every element of $Y$ is of the form $\gamma \basepointy$ for some $\gamma\in\GL(2,\Q)$. In addition, given $y,y'\in Y$, there is a unique $\gamma\in\GL(2,\Q)$ with $y=\gamma y'$.  There is a natural action of $K^\times$ by scalar multiplication on $Y$, which we write as a right action.

\begin{definition} Let $y=\transpose{(a,b)}\in Y$.  Define $\Lambda_y$ to be the $\Z$-lattice in $K$ generated by $a$ and $b$ (i.e.\ the set of all integer linear combinations of $a$ and $b$).
\end{definition}

Note that for $\alpha\in K^\times$, we have $\Lambda_{y\alpha }=\alpha\Lambda_y$.

\begin{definition} Let $H\subseteq K^\times$ be a multiplicative subgroup of $K^\times$.  Two lattices $\Lambda_1$ and $\Lambda_2$ in $K$ will be said to be {\em homothetic} if there is some $\alpha\in K^\times$ such that $\Lambda_1=\alpha\Lambda_2$.  If $\alpha\in H$, we will say that the lattices are {\em $H$-homothetic}.
\end{definition}

Homothety and $H$-homothety of lattices are equivalence relations on the set of all lattices in $K$.  

\begin{definition}
Let $H$ be a multiplicative subgroup of $K^\times$.  Define $Y/H$ to be the quotient of $Y$ with respect to the right action of scalar multiplication by $H$.  The left action of $\gl(2,\Q)$ on $Y$ then gives a left action of $\gl(2,\Q)$ on $Y/H$.
\end{definition}

\begin{lemma} There is a bijection between $\gl(2,\Z)$-orbits of elements of $Y/H$ and $H$-homothety classes of lattices in $K$.
\end{lemma}

\begin{proof}
Denote the set of $H$-homothety classes of lattices in $K$ by $\homothety$. Define a map
\[f:Y/H\to\homothety\]
by setting $f(x)$ equal to the $H$-homothety class of $\Lambda_y$ for any $y\in Y$ representing $\xS\in Y/H$.  Since the various $y$ representing $x$ all differ by scalar multiples by some element of $H$, it is clear that $f$ is well defined.  Since every lattice in $K$ is of the form $\Lambda_y$ for some $y\in Y$, the map $f$ is surjective. Finally, for $\gamma\in\GL(2,\Z)$, we have that $\gamma x$ is represented by $\gamma y$, and that $\Lambda_y=\Lambda_{\gamma y}$.  Hence, $f$ is constant on $\GL(2,\Z)$-orbits, and so induces a surjective map $\hat f$ from the set of $\GL(2,\Z)$-orbits in $Y/H$ to $\homothety$.

Now suppose that $x,x'\in Y/H$ are represented by $y,y'\in Y$, and $f(x)=f(x')$.  Then $\Lambda_y=\alpha\Lambda_{y'}=\Lambda_{y'\alpha}$ for some $\alpha\in H$.  Hence, the entries of both $y$ and $y'\alpha$ are a basis for $\Lambda_y$.  Therefore, there is some $\gamma\in\gl(2,\Z)$ such that $y=\gamma y'\alpha$.  Then $x=\gamma x'$, so $x$ and $x'$ are in the same $\GL(2,\Z)$-orbit.  Hence, $\hat f$ is an injective map on $\GL(2,\Z)$-orbits.
\end{proof}

\begin{lemma}\label{Lemma:mL} Let $\Lambda$ be a lattice in $K$.  Then there is a minimal positive integer $m$ such that $\epsilon^m\Lambda=\Lambda$.
\end{lemma}
\begin{proof} 
Note that if $\Lambda$ and $\Lambda'$ are $K^\times$-homothetic, the lemma will be true for $\Lambda$ if and only if it is true for $\Lambda'$, with the same value of $m$ (since $K^\times$ is commutative.)  Hence, we may, without loss of generality, assume that $\Lambda$ is contained in $\OK$.  Since $\Lambda$ is a rank two $\Z$-submodule of $\OK$, it must have finite index in $\OK$.  We may thus choose an $N\in\Z$ such that $N\OK\subseteq\Lambda\subseteq\OK$.  Since $\OK/N\OK$ is finite and multiplication by $\epsilon$ permutes its elements, there is some positive $m\in\Z$ such that $\delta=\epsilon^m$ acts trivially on $\OK/N\OK$, and hence on $\Lambda/N\OK$.  Then $\delta$ must take $\Lambda$ to itself, so $\delta\Lambda\subseteq\Lambda$.  We must also have $\delta^{-1}\Lambda\subseteq\Lambda$, so 
$\Lambda\subseteq\delta\Lambda\subseteq\Lambda$,
and therefore $\delta\Lambda=\Lambda$. The existence of a minimal positive $m$ satisfying the conditions of the theorem follows immediately from the existence of some positive $m$.
\end{proof}

\begin{definition} Given $\xS\in Y/H$, we define $\stab{\xS}$ to be the stabilizer of $\xS$ in $\gl(2,\Z)$, and $\stabhat{\xS}$ to be the quotient
\[\frac{\stab{\xS}}{\stab{\xS}\cap\{\pm I\}}.\]
\end{definition}

\begin{definition} We will say that a subgroup $H$ of $K^\times$ is {\em unit-cofinite}
if $H\cap\OK^\times$ has finite index in $\OK^\times$.
\end{definition}

\begin{theorem}\label{stabgen}
Let $H$ be a unit-cofinite subgroup of $K^\times$, and let $\xS\in Y/H$ be represented by $y\in Y$.  Then $\stabhat{\xS}$ is a cyclic group, generated by the image of the unique element $g\in\stab{\xS}$ satisfying
\[gy=y\delta\]
where $\delta=\pm\epsilon^m$, and $m$ is smallest possible positive integer such that $\Lambda_y=\Lambda_{y\epsilon^m}$ and $\delta\in H$.
\end{theorem}

\begin{remark} 
The notation in the theorem means that $\delta=\epsilon^m$ or $\delta=-\epsilon^m$, depending on which is in $H$.  If both are in $H$, then $-1\in H$ and we set $\delta=\epsilon^m$.
If $-1\notin H$, it is possible to have $-\epsilon^m\in H$ without having $\epsilon^m\in H$.  Hence, it is necessary to choose $\delta=\pm\epsilon^m$ with $m$ minimal to get a generator of $\stab{x}$.
\end{remark}

\begin{proof} Let $x$ be represented by $y=\transpose{(a,b)}\in Y$.  Choose the smallest positive $m$ such that $\epsilon^m\Lambda_y=\Lambda_y$ and one (or both) of $\epsilon^m$ and $-\epsilon^m\in H$.  Let $\delta=\pm\epsilon^m\in H$.  Then $\delta\Lambda_y=\Lambda_y$, so $y\delta$ is a basis of $\Lambda_y$.  Hence, there is some $g\in\gl(2,\Z)$ such that $gy=y\delta$.  Since $\delta\in H$, we see that $g\xS=\xS$.

We now show that every element in $\stab{\xS}$ is (up to a sign) a power of $g$.  Let $\eta\in\stab{\xS}$.  Then, since $\eta \xS=\xS$, there is some $\alpha\in H$ such that $\eta y=y\alpha$. Now $\alpha$ is an eigenvalue of $\eta$, and $\eta\in\gl(2,\Z)$, so $\alpha\in\OK^\times$.  Hence, $\alpha=\pm\epsilon^r$.  By the division algorithm and the minimality of $m$, we see that $\alpha=\pm\delta^k$ for some $k$.  Hence, $\eta=\pm g^k$.
If $\eta= g^k$ we are finished.  If $\eta=- g^k$ then $-I\in\Gamma_x$ and $\eta\equiv g^k$ modulo $\Gamma_x\cap\{\pm I\}$.
\end{proof}

Certain elements $\xS\in Y/H$ will be quite important to us; for these elements, the value of $m$ in the previous proof is determined solely by $H$.

\begin{definition} Let $H$ be a multiplicative subgroup of $K^\times$.  If $\xS\in Y/H$ can be represented by $y\in Y$ such that $\Lambda_y$ is a fractional ideal in $K$, then we say that $\xS$ is {\em idealistic}.
\end{definition}

Note that determining whether $\xS$ is idealistic does not depend on the choice of $y\in Y$ representing $\xS$.

\begin{corollary} Let $H$ be a unit-cofinite subgroup of $K^\times$.  If $\xS\in Y/H$ is idealistic, the value of $m$ in Theorem~\ref{stabgen} is equal to the smallest positive integer $k$ such that $\pm\epsilon^k\in H$.
\end{corollary}

\begin{definition}\label{D:generator} Let $H$ be a unit-cofinite subgroup of $K^\times$.  For $\xS\in Y/H$, denote the positive integer $m$ described in Theorem~\ref{stabgen} by $m_x$, and the 
element $g$ described in Theorem~\ref{stabgen} by $g_x$.
\end{definition}

\begin{corollary}\label{gammaxcor} Let $H$ be a unit-cofinite subgroup of $K^\times$.  If $x,x'\in Y/H$ are in the same $\GL(2,\Z)$-orbit, then  $m_x=m_{x'}$.
\end{corollary}

\begin{proof} If $x=\gamma x'$ for $\gamma\in\GL(2,\Z)$, then $\Gamma_{x'}=\gamma\Gamma_x\gamma^{-1}$, so $g_x$ is conjugate to $g_{x'}$.  Then $g_x$ and $g_{x'}$ have the same eigenvalues, so $x$ and $x'$ have the same value of $m$.
\end{proof}

Assume that $H$ does not contain $-1$.  In this case, $\stab{x}$ does not contain $-I$, so $\stabhat{x}=\stab{x}$ is cyclic, generated by $g_x$.  Then, there is a canonical isomorphism
\[I_x:H_1(\stab{x},\F)\to\stab{x}\otimes_\Z\F.\]

\begin{definition}\label{Def:zx}
If $-1\notin H$, and $\xS\in Y/H$, define $z_x$ to be the generator of $H_1(\stab{x},\F)$ such that $I_x(z_x)=g_x\otimes 1$.
\end{definition}

\section{$(S,q)$-homotheties}

From now on, we fix a field $\F$ of characteristic not equal to $2$.

\begin{definition} Let $y=\transpose{(a,b)}\in Y$.  We define an injective homomorphism $r_y:K^\times\to\gl(2,\Q)$ by
\[r_y(c)\column{a}{b}=\column{ac}{bc}\]
for $c\in K^\times$.
\end{definition}




\begin{definition} Let $M$ be any positive integer.  Define $S_0(M)$ to be the largest subgroup of $\GL(2,\Q)$ that can be mapped modulo $M$ to $\gl(2,\Z/M\Z)$.  Define $S(M)$ to be the kernel of reduction modulo $M$ from $S_0(M)$ to $\gl(2,\Z/M\Z)$.
\end{definition}

We note that for any $M$, $S_0(M)$ contains $\gl(2,\Z)$.
  
\begin{definition}\label{D:S_0} Recall that $d=\disc(K)$.  If the characteristic of $\F$ is nonzero, set $p$ equal to the characteristic of $F$; otherwise, set $p=1$.  Fix positive integers $M,N$ with $M\geq 3$ such that $M\mid pdN$.  Define $S_0=S_0(pdN)$ and $S=S(M)\cap S_0(pdN)$.
\end{definition}

Since $S$ is the kernel of the composition
\[S_0\to\gl(2,\Z/pdN\Z)\to\gl(2,\Z/M\Z),\]
we see that $S$ has finite index in $S_0$.  Further, since $M\geq3$, $-I\notin S$.

\begin{definition}  Recall that $\basepointy=\transpose{(\omega,1)}\in Y$. Define
\[K_{S_0}=\{c\in K^\times:r_\basepointy(c)\in S_0\}\]
and
\[K_S=\{c\in K^\times:r_\basepointy(c)\in S\}.\]
We note that $K_{S_0}$ and $K_S$ are multiplicative subgroups of $K^\times$.
\end{definition}

Since $S\subseteq S_0$ has finite index, it is clear that $K_S\subseteq K_{S_0}$ with finite index.  

\begin{lemma}
If $\alpha\in K_{S_0}$, then $\alpha$ is relatively prime to $pdN$.
\end{lemma}
\begin{proof} Since $\alpha\in K$, we may write $\alpha=\beta/n$ with $\beta\in\OK$ and $n\in\Z$.  Since $\beta\in\OK$, we see that $r_\basepointy(\beta)$ has integer entries.  If a prime $b|dN$ and $b|n$, then, since the entries of $r_\basepointy(\alpha)=\frac1nr_\basepointy(\beta)$ can have no denominators divisible by $b$,  it must be the case that $b$ divides every entry of $r_\basepointy(\beta)$.  This implies that $b|\beta$ in $\OK$.  Canceling (repeatedly, if needed), we may take $n$ to be relatively prime to $pdN$.

Now $N^K_\Q(\alpha)=\det(r_\basepointy(\alpha))=\det(r_\basepointy(\beta))/n^2=N^K_\Q(\beta)/n^2$ must be relatively prime to $pdN$, so $N^K_\Q(\beta)$ is also relatively prime to $pdN$.  Since $\beta\in\OK$, it must then be relatively prime to $pdN$, and so $\alpha$ is as well.
\end{proof}

\begin{lemma} $K_S$ is a unit-cofinite subgroup of $K^\times$.  
\end{lemma}

\begin{proof} We note that $r_\basepointy(\epsilon)\in\gl(2,\Z)$.  Hence, $r_\basepointy(\epsilon)$ can be reduced modulo $M$ to give a matrix in $\gl(2,\Z/M\Z)$.  Since $\gl(2,\Z/M\Z)$ is finite, there is some positive integer $k$ such that 
\[r_\basepointy(\epsilon)^k=r_\basepointy(\epsilon^k)\in S(M).\]
Since $r_\basepointy(\epsilon)$ is also in $S_0$, we see that $\epsilon^k\in K_S$, so $K_S$ is unit-cofinite.
\end{proof}

\begin{definition}\label{D:q*} Let $Z$ denote the set of scalar matrices $\zeta_r=rI$ where $r\in \Q^\times\cap K_{S_0}$.  
Let $\theta:\Z\to\F^\times$ 
be the quadratic Dirichlet character cutting out $K$, and define $q:Z\to\F^\times$ by $q(rI)=\theta(r)$ for $r\in\Z\cap K_{S_0}$, extended multiplicatively to $Z$.

We define $ q^*:K_{S_0}\to\F^\times$ to be the composition of the following multiplicative maps:
\begin{enumerate}[(1)]
\item The map taking $a\in K_{S_0}$ to the principal fractional ideal $(a)\subset K$,
\item The map taking a fractional ideal to its prime factorization,
\item The map taking a product of powers of prime ideals to the subproduct of powers of inert prime ideals,
\item The map taking an inert prime ideal $(\ell)$ to $q(\ell I)$.
\end{enumerate}

We note that $q^*$ is a homomorphism, and we define $\KSQ$ to be the kernel of $q^*|_{K_S}$.  Then $\KSQ$ is a multiplicative subgroup of $K^\times$ and $\KSQ$ has index $2$ in $K_S$.
\end{definition}

\begin{lemma}\label{L:unit-cofinite} 
$\KSQ$ is a unit-cofinite subgroup of $K^\times$.  
\end{lemma}

\begin{proof} 
This is true because  $K_S$ is unit-cofinite, and any unit in $K_S$ is in the kernel of $q^*$, and so in $\KSQ$.
\end{proof}

\begin{lemma}\label{Lemma:q*-q}
For $r\in\Q^\times\cap K_{S_0}$, $q(rI)=q^*(r)$.
\end{lemma}
\begin{proof}
Let 
\[r=\prod_{\ell\text{ prime}}\ell^{n_\ell}.\]
Then 
\[q(rI)=\theta(r)=\prod_{\ell\text{ prime}}\theta(\ell)^{n_\ell}=\prod_{\ell\text{ inert in $K$}}\theta(\ell)^{n_\ell}=\prod_{\ell\text{ inert in $K$}} q(\ell I)^{n_\ell}=q^*(r).\]
since $n_\ell=0$ for $\ell$ ramified in $K/\Q$, and $\theta(\ell)=1$ for $\ell$ split in $K$.
\end{proof} 

Let $\XSQ=Y/\KSQ$.  Then $\F Y$ is an $(S_0,\KSQ)$-bimodule, and we obtain an isomorphism
\[\F Y\otimes_{\KSQ}\F\cong \F\XSQ\otimes \F=\F\XSQ\]
of $S_0$-modules.

\begin{lemma}\label{L:X^S-stab}
Let $\xS\in\XSQ$ be represented by $y\in Y$.  Then
\begin{enumerate}[(a)]
\item $\stab{\xS}=\{r_y(c):c\in\KSQ\}\cap\gl(2,\Z).$
\item If $g\in\stab{\xS}$, then $g=r_y(c)$ for some $c\in\OK^\times$.
\item $-I\notin \stab{x}$ and $\stab{\xS}$ is infinite cyclic.
\end{enumerate} 
\end{lemma}

\begin{proof}
(a)
 Suppose $g\in\stab{x}$.  Then we have that $gy=yc$ for some $c\in \KSQ$.  Since $yc=r_y(c)y$, and the entries of $y$ are a $\Q$-basis of $K$, we see that $g=r_y(c)$.  Hence, $g$ is in the given intersection, and any $g$ in the given intersection fixes $\xS$.

(b)  Let $g\in\stab{x}$.  Then $g=r_y(c)$, for some $c\in\KSQ$, and the characteristic polynomial of $g$ is the same as that of multiplication by $c$ on $K$.  Since $g\in\gl(2,\Z)$, we see that $c\in \OK^\times$.

(c) This follows from Theorem~\ref{stabgen}, Lemma~\ref{L:unit-cofinite}, and the fact that $-1\notin K_{S,q}$.
\end{proof}

Denote by $\basepoint$ the image in $\XSQ$ of $\basepointy$.

\begin{definition}
 Define $i^S=m_{\basepoint}$.
\end{definition}

\begin{lemma}\label{is}
Let $\xS\in\XSQ$.  
\begin{enumerate}[(i)]
\item For any $x$, $i^S\mid m_x$.
\item If $x$ is idealistic, then $m_x=i^S$.
\end{enumerate}
\end{lemma}

\begin{proof}
(i)
For any $\xi\in \XSQ$, let $\phi_\xi:\Gamma_\xi \to \KSQ$ be the injective homomorphism defined by
$g=r_\xi(\phi_\xi(g))$ for $g\in\Gamma_\xi$.  Then the image of $\phi_\xi$ is generated by $\epsilon^{m_\xi}$.
 
 Now let $\xS\in\XSQ$ be represented by $y\in Y$. We have seen that any element of $\stab{\xS}$ is of the form $r_y(c)$ for some $c\in\OK^\times$.
 From Lemma~\ref{L:X^S-stab}, we see that for $c\in\OK^\times$,
\[r_y(c)\in\stab{\xS}\implies c\in \KSQ\iff r_{\basepoint}(c)\in\stab{\basepoint}.\]
Hence, the image of $\phi_x$ is contained in the image of $\phi_\basepoint$.  Since both images are cyclic groups, we see that $i^S|m_x$.

(ii)  Since $x$ is idealistic, $\Lambda_y$ is a fractional ideal of $K$.  Hence, $r_y(c)\in\GL(2,\Z)$ for all $c\in\OK^\times$.  Therefore, we see that for $c\in\OK^\times$,
\[r_y(c)\in\stab{x}\iff c\in K_{S,q}\iff r_{\basepoint}(c)\in\stab{\basepoint}.\]
Clearly, then, $m_x=i^S$.
\end{proof}

\begin{definition}\label{D:mx'} 
For $\xS\in\XSQ$, set $m_x'=m_x/i^S$.
\end{definition}

Finally, we prove the following lemma about the relationship between elements of $Z$ and elements of $\KSQ$.

\begin{lemma}\label{Lemma:Z-KSQ}
Let $\zeta=rI\in\Z$ with $r\in\Q^\times\cap K_{S_0}$, let $\alpha\in\KSQ$, and let $y\in Y$.  If $\zeta r_y(\alpha)\in\gl(2,\Z)$, then $q(\zeta)=1$.
\end{lemma}
\begin{proof}
We have that $\zeta r_y(\alpha)=r_y(r\alpha)\in\gl(2,\Z)$.  Since the characteristic  polynomial of $r\alpha$ is the same as the characteristic polynomial of $r_y(r\alpha)$, we see that $r\alpha$ is a unit in $\OK$.  Hence, $q^*(r\alpha)=1$.  Since $q^*(\alpha)=1$, by the multiplicativity of $q^*$ we see that $q^*(r)=1$. By Lemma~\ref{Lemma:q*-q}, it follows that $q(rI)=q(\zeta)=1$.
\end{proof}

\section{Defining the coefficient module $\MSQ$}\label{MSQ}

\begin{definition}\label{D:MSQ}
Define $\MSQ$ to be the $\F$-vector space of formal sums
\[\sum_{\xS\in\XSQ} c_{\xS}\xS\]
with $c_\xS\in \F$, such that the sum is supported on a finite number of $Z$-orbits of $\XSQ$, and such that the coefficients satisfy the relation
\[c_{\zeta\xS}=q^{-1}(\zeta)c_{\xS}\]
for all $\zeta\in Z$.
\end{definition}

In this paper $q$ has order $2$, but we write $q^{-1}$ as a check on our computations and for possible generalizations to characters of higher orders.

\begin{definition} 
Define $Z_{S,q}=\{\zeta_r\in Z:r\in \KSQ\}$. 
\end{definition} 

Note that $Z_{S,q}$ is a subgroup of finite index in $Z$ (since $\KSQ$ has finite index in $K_{S_0}$).  Let $\zreps$ be a collection of coset representatives of $Z_{S,q}$ inside $Z$.

Now $Z\,\gl(2,\Z)$ is a group which acts on $\XSQ$.  Hence, we may choose a collection of representatives of the $Z\,\gl(2,\Z)$-orbits in $\XSQ$.  We will denote such a collection by $\A$.

Note that given any $x\in\XSQ$, we may assume (possibly by changing $\A$) that $x\in \A$.  For the remainder of this section, we fix a set $\A$ of $Z\,\gl(2,\Z)$-orbit representatives.

Clearly each $Z$-orbit in $\XSQ$ contains at least one element of the form $gx$ with $g\in\gl(2,\Z)$ and $x\in\A$.  

\begin{definition}\label{D:Z-orbit-reps}
Let $\zorbitreps$ be a collection of representatives of the $Z$-orbits in $\XSQ$, chosen so that each representative 
in $\zorbitreps$ is of the form $gx$ for $g\in\gl(2,\Z)$ and $x\in\A$.
\end{definition}

\begin{remark}\label{Remark:z-orbits} Note that $\zorbitreps$ is not uniquely determined by $\A$.  However, once a choice of $\zorbitreps$ is fixed, any $Z$-orbit will contain a unique representative $gx\in\zorbitreps$, and the element $x\in \A$ and the coset $g\stab{x}$ of $g\in\gl(2,\Z)$ are uniquely defined. 
\end{remark}

For the remainder of this section, we fix a choice of $\zorbitreps$ corresponding to our choice of $\A$.

\begin{lemma}\label{msqbasis} The set
\[\left\{\sum_{\zeta\in\zreps}q^{-1}(\zeta)\zeta x:x\in\zorbitreps\right\}\]
is an $\F$-basis of $\MSQ$.  
\end{lemma}

\begin{proof} For any $x\in\XSQ$, the $Z$-orbit of $x$ is equal to the set $\{\zeta x:\zeta\in\zreps\}$.  By the relation on the coefficients of an element in $\MSQ$, the coefficient of $\zeta x$ is equal to $q^{-1}(\zeta)$ times the coefficient of $x$.  Since an element of $\MSQ$ is supported on finitely many $Z$-orbits, the lemma follows.
\end{proof}

\begin{lemma}\label{L:Z-action} $\MSQ$ is an $S_0$-module.  For $r\in\Q^\times\cap K_{S_0}$, the action of $\zeta_r$ on $\MSQ$ is via the scalar $q(\zeta_r)$.
\end{lemma}

\begin{proof}
For $s\in S_0$, we have $s\sum c_\xS\xS=\sum c_{\xS} s\xS\in\MSQ$, since $Z$ is in the center of $S_0$.

It suffices to prove the statement about the action of $\zeta_r$ on basis elements of the form
\[\sum_{\zeta_\ell\in\zreps} q^{-1}(\zeta_\ell)\zeta_\ell\xS,\]
with $\xS\in\zorbitreps$. For $\zeta$ in $Z$, we will define $\bar\zeta$ to be the unique element of $\zreps$ such that $\zeta\in\bar\zeta \ZSQ$.  The map from $\zreps$ to $\zreps$ given by $\zeta\mapsto\overline{\zeta\zeta_r}$ for a fixed 
$\zeta_r$ is a bijection.  We also note that $q(\overline\zeta)=q(\zeta)$ for any $\zeta\in Z$.

Setting $u=r\ell$, we now have 
\begin{align*}
\zeta_r\sum_{\zeta_\ell\in\zreps} q^{-1}(\zeta_\ell)\zeta_\ell\xS&=\sum_{\zeta_\ell\in\zreps} q^{-1}(\zeta_\ell)\zeta_\ell\zeta_r\xS\cr
&=\sum_{\zeta_\ell\in\zreps} q^{-1}(\zeta_\ell)\zeta_{\ell r}\xS\cr
&=\sum_{\overline\zeta_u\in\zreps} q^{-1}(\zeta_u\zeta_r^{-1})\zeta_u\xS\cr
&=q^{-1}(\zeta_r^{-1})\sum_{\overline\zeta_u\in\zreps} q^{-1}(\overline\zeta_u)\overline\zeta_u\xS\cr
&=q(\zeta_r)\sum_{\zeta_\ell\in\zreps} q^{-1}(\zeta_\ell)\zeta_\ell\xS.\qedhere
\end{align*}
\end{proof}

\begin{corollary}\label{C:unique-basis} The basis described in Lemma~\ref{msqbasis} is independent of the choice of $\zorbitreps$.
\end{corollary}

\begin{proof} Let $\zorbitreps$ and $\zorbitreps'$ be two choices of $Z$-orbit representatives, as in Definition~\ref{D:Z-orbit-reps}.  For a given $Z$-orbit, let $gx\in\zorbitreps$ and $g'x\in\zorbitreps'$ (with $g,g'\in\gl(2,\Z)$) be the orbit representatives. Note that they will have the same representative $x\in\A$, since they are in the same $Z\,\gl(2,\Z)$-orbit.  Since $gx$ and $g'x$ are in the same $Z$-orbit, for some $\zeta\in Z$ we have $\zeta g'x= gx$, so that $\zeta g^{-1}g'x=x$, and (choosing a $y\in Y$ representing $x$) we see that $\zeta g^{-1}g'y\alpha=y$ for some $\alpha\in\KSQ$. Hence, $\zeta r_y(\alpha)\in\GL(2,\Z)$, and by Lemma~\ref{Lemma:Z-KSQ} we see that  $q(\zeta)=1$.  The corresponding basis elements then differ by a factor of $\zeta$, so that by Lemma~\ref{L:Z-action} they differ by a scalar factor of $q(\zeta)=1$.
\end{proof}
 
\begin{lemma}\label{induced} If we consider $\MSQ$ as a $\gl(2,\Z)$-module, it is a sum of induced modules.  In fact, we have isomorphisms of $\gl(2,\Z)$-modules defined by
\[e:\bigoplus_{\xS\in\mathcal A} \F [\gl(2,\Z)]\otimes_{\F \stab{\xS}}\F \to\MSQ,\]
and
\[f:\MSQ\to\bigoplus_{\xS\in\mathcal A} \F[\gl(2,\Z)]\otimes_{\F \stab{\xS}}\F, \]
such that $e$ and $f$ are inverses of each other.
\end{lemma}

\begin{remark} These isomorphisms depend on the choice of $\A$, which we suppress from the notation.
\end{remark}

\begin{proof}
On basis elements of the form $g\otimes 1\in \F [\gl(2,\Z)]\otimes_{\F\stab{\xS}}\F$ for $\xS\in\mathcal A$, we define $e$ as
\[e(g\otimes 1)=\sum_{\zeta_r\in\zreps}q(\zeta_r^{-1})\zeta_r g\xS,\]
and extend linearly.  Since any $g\in\stab{\xS}$ acts trivially on $\xS$, this is well-defined, and it is clearly a homomorphism of $\gl(2,\Z)$-modules.  
 
We define $f$ on basis elements corresponding to $gx\in\zorbitreps$ (with $x\in\A$ and $g\in\gl(2,\Z)$) by
\[f\left(\sum_{\zeta\in\zreps}q^{-1}(\zeta)\zeta gx\right)=g\otimes1\in\F\,\gl(2,\Z)\otimes_{\F\,\stab{x}}\F\subset\bigoplus_{x\in\A}\F\,\gl(2,\Z)\otimes_{\F\,\stab{x}}\F ,\]
and extending linearly.  By Lemma~\ref{msqbasis}, this gives a well defined $\F$-linear map from $\MSQ$ to 
\[\bigoplus_{x\in\A}\F\,\gl(2,\Z)\otimes_{\F\,\stab{x}}\F.\]

One sees that on basis elements, $e$ and $f$ are inverses.  Hence, they are inverses of each other as $\F$-vector space maps.  Since $e$ is a $\gl(2,\Z)$-module map, so is $f$ and thus both are $\gl(2,\Z)$-module isomorphisms.
\end{proof}

\section{Homology with coefficients in $\MSQ$ and Hecke operators}\label{coefficients}

In this section, we will fix an element $x_0\in\XSQ$, and choose a set $\A$ of $Z\,\gl(2,\Z)$-orbit representatives in $\XSQ$ that contains $x_0$.

As a consequence of Lemma~\ref{induced}, we have that
\[\MSQ\cong \bigoplus_{\xS\in\mathcal A} \ind_{\stab{\xS}}^{\GL(2,\Z)}\F.\]
Hence, by Shapiro's lemma, we have
\[H_1(\gl(2,\Z),\MSQ)\cong \bigoplus_{\xS\in\mathcal A}H_1(\Gamma_{\xS},\F).\]
Since $\stab{\xS}$ is infinite cyclic, we have
\[H_1(\stab{\xS},\F)\cong H^0(\stab{\xS},\F)\cong \F.\]
For each $\xS\in \mathcal A$, we choose a generator $\zxs$ for $H_1(\stab{\xS},\F)$, as in Definition~\ref{Def:zx}.

We now examine an individual Hecke operator.  Let $s\in S_0$, and let $E$ be a collection of single coset representatives for $\gl(2,\Z)s\gl(2,\Z)$, so that
\[\gl(2,\Z)s\gl(2,\Z)=\coprod_{s_\alpha\in E} \gl(2,\Z)s_\alpha.\]
Then $E$ is a finite set.  At this point, the $s_\alpha$ may be altered by left-multiplication by elements of $\gl(2,\Z)$.  We now adjust the elements of $E$ to make the computation of Hecke operators easier.

Because of our choice of $\A$, we have that $x_0\in\A$.  For convenience in what follows, we will write $\stab{\xS_0}=\stab0$.  Recall that $\mathcal A$ is a collection of $Z\,\gl(2,\Z)$-representatives of $\XSQ$.  Hence, we may write any element of $\XSQ$ as $\zeta\gamma \xS$ for some $\xS\in\A$, $\zeta\in \zreps$ and $\gamma\in\gl(2,\Z)$.  Suppose that for $s_\alpha\in E$, we have
\[s_\alpha\xS_0=\zeta\gamma\xS.\]
We will then adjust $s_\alpha$, replacing it by $\gamma^{-1}s_\alpha$, and denote the corresponding $\zeta$ by $\zeta_\alpha$, and the corresponding $\xS$ by $\xS_\alpha$, so that we have
\[s_\alpha\xS_0=\zeta_\alpha\xS_\alpha.\]
We now fix this choice of $E$.

For a given $\zeta\in Z$ and $\xS\in\A$, we define
\[E(\zeta\xS)=\{s_\alpha\in E:s_\alpha\xS_0=\zeta\xS\}.\]
Since each $s_\alpha x_0$ is of the form $\zeta \xS$, we see easily that $E$ is a disjoint union of the $E(\zeta\xS)$ as $\zeta$ runs through $\zreps$ and $\xS$ runs through $\A$ (and $E(\zeta\xS)$ is empty for all but finitely many
 $\xS\in\A$).

Now, let $\zeta\in\zreps$, $\xS\in\A$, and (to avoid triviality) assume that for some $s_\alpha\in\A$, we have $s_\alpha\xS_0=\zeta\xS$.  We will then define $W_{\zeta\xS}$ to be the set of all elements $g\in\gl(2,\Z)s\gl(2,\Z)$ such that $g\xS_0=\zeta\xS$.  This set is nonempty, and stable under right multiplication by $\stab{0}$ and under left multiplication by $\stab{\xS}$.  Hence, we may write it as a disjoint union of double cosets
\[W_{\zeta\xS}=\coprod_{t\in T_{\zeta\xS}}\stab{\xS}t\stab0,\]
for some subset $T_{\zeta\xS}\subseteq W_{\zeta\xS}$.  

For each $t\in T_{\zeta\xS}$, 
we may choose a set $B_{\zeta x,t}$ of single coset representatives and write the double coset $\stab{\xS}t\stab0$ as a disjoint union of single cosets
\[\stab{\xS}t\stab0=\coprod_{t_\beta\in B_{\zeta\xS,t}}\stab{\xS}t_\beta.\]

\begin{lemma}\label{L:partition} With the notation described above,
\begin{enumerate}
\item $E$ is a disjoint union of the $E(\zeta\xS)$ for $\zeta\in\zreps$ and $\xS\in\A$.
\item For each $t\in T_{\zeta\xS}$, we may choose $B_{\zeta\xS,t}$ to be a subset of $E$.
\item With this choice, $E(\zeta\xS)$ is the disjoint union of the $B_{\zeta\xS,t}$ over all $t\in T_{\zeta\xS}$,
and $B_{\zeta\xS,t}=\stab{\xS}t\stab0\cap E(\zeta\xS)$.
\end{enumerate}
\end{lemma}
\begin{proof} First, note that because $\zeta$ is central,  $\stab{\zeta\xS}=\stab{\xS}$ for any $\xS\in \XSQ$.

We have seen (1) above.  

For (2), choose $t\in T_{\zeta\xS}$, and let $\stab{\xS}u$ be any single coset in $\stab{\xS}t\stab0$.  Then $u\xS_0=\zeta\xS$, and $u$ is in some single coset $\GL(2,\Z)s_\alpha$ for some $s_\alpha\in E$.  Then $u=gs_\alpha$ for some $g\in\GL(2,\Z)$.  We then have that
\[\zeta\xS=u\xS_0=gs_\alpha\xS_0=g\zeta_\alpha\xS_\alpha.\]
This implies that $\xS$ and $\xS_\alpha$ are in the same $Z\,\gl(2,\Z)$-orbit, and hence equal (since both come from $\A$).  Hence, $g\in\stab{\xS}$, and consequently $\zeta_\alpha x=\zeta x$.  
Therefore $\zeta_\alpha\zeta^{-1}\in \ZSQ$ and since $\zeta_\alpha,\zeta\in\zreps$, it follows that $\zeta_\alpha=\zeta$.
Hence, we have $s_\alpha\in E(\zeta\xS)$, and $\stab{\xS}u=\stab{\xS}s_\alpha$, so we see that we may take the coset representative of $\stab{\xS}u$ to be $s_\alpha\in E$.

For (3), we first note that any coset  $\stab{\xS}t$ for $t\in T_{\xS}$ contains exactly one $s_\alpha$:  
part (2) shows that it contains at least one; if it contained two, say $s_\alpha$ and $s_\delta$, then they would differ by left multiplication by $\gamma\in\stab{\xS}\subseteq\gl(2,\Z)$, which would imply $\gl(2,\Z)s_\alpha=\gl(2,\Z)s_\delta$ and therefore $s_\alpha=s_\delta$.
 Hence, it suffices to show that each $s_\alpha\in E(\zeta\xS)$ is contained in $B_{\zeta\xS,t}$ for some $t\in T_{\zeta\xS}$.  This is, however, clear, since such an $s_\alpha$ is contained in
\[W_{\zeta\xS}=\bigcup_{t\in T_{\zeta\xS}}\stab{\xS}t\stab0.\]
The last assertion is now clear.
\end{proof}

From this point on, we will take $B_{\zeta\xS,t}=\stab{\xS}t\stab{0}\cap E(\zeta\xS)$.

We continue to keep a fixed $\xS_0$.  As $\zeta\in\zreps$, $\xS\in\A$, and $t\in T_{\zeta\xS}$ vary, finitely many $B_{\zeta\xS,t}$ will be nonempty.  We denote these sets by $B_1,\ldots,B_J$ for $J\in\Z$ positive, and for $j\in\{1,\ldots,J\}$ we write $\xS_j$, $\zeta_j$, and $t_j$ for the corresponding values of $\zeta$, $\xS$, and $t$, respectively.  In addition, we will write $\stab j$ for $\stab{\xS_j}$.
With this notation, we now have
\[t_j\xS_0=\zeta_j\xS_j\quad\text{and}\quad s_\alpha\xS_0=\zeta_j\xS_j\]
for all $s_\alpha\in B_j$.

We note that for $s_\alpha\in B_j$, there exist $\eta_\alpha\in\stab j$ and $\delta_\alpha\in\stab0$ such that $s_\alpha=\eta_\alpha t_j\delta_\alpha$.  

\begin{lemma}\label{L:stab-reps} Fix $j\in\{1,\ldots,J\}$, and for each $s_\alpha\in B_j$, write $s_\alpha=\eta_\alpha t_j \delta_\alpha$ with $\delta_\alpha\in\stab0$ and $\eta_\alpha\in\stab j$.  Let $d_j=|B_j|$.  Then
\[\stab0=\coprod_{s_\alpha\in B_j}(t_j^{-1}\stab jt_j\cap\stab0)\delta_\alpha\]
so that we have
\[[\stab0:t_j^{-1}\stab jt_j\cap\stab0]=d_j.\]
\end{lemma}

\begin{proof} 
First we show that the displayed cosets are all different.
Suppose $s_\alpha,s_\beta\in B_j$ with $\delta_\alpha \delta_\beta^{-1}=t_j^{-1}\gamma_jt_j\in t_j^{-1}\stab jt_j\cap\stab0$ for some $\gamma_j\in\stab j$.  Then $\delta_\alpha=t_j^{-1}\gamma_jt_j\delta_\beta$.  Hence
\begin{align*}
\stab js_\alpha&=\stab j \eta_\alpha t_j\delta_\alpha\cr
&=\stab j \eta_\alpha t_j(t_j^{-1}\gamma_jt_j\delta_\beta)\cr
&=\stab j t_j\delta_\beta\cr
&=\stab j \eta_\beta t_j\delta_\beta\cr
&=\stab j s_\beta,
\end{align*}
where we have used that $\eta_\alpha,\eta_\beta,\gamma_j\in\stab j$.  Hence, $s_\alpha=s_\beta$, and we see that the cosets $(t_j^{-1}\stab jt_j\cap\stab0)\delta_\alpha$ are pairwise disjoint for $s_\alpha\in B_j$.

It remains only to show that the union of the cosets is all of $\stab0$.  Let $g\in\stab0$.  Since $\stab jt_j\stab0$ is a disjoint union of cosets  of the form $\stab j s_\alpha$ for some $\alpha\in B_j$, we have that $t_jg=\gamma_js_\alpha$ for some $\gamma_j\in\stab j$ and $s_\alpha\in B_j$.  Then, since $s_\alpha=\eta_\alpha t_j\delta_\alpha$, for $\eta_\alpha\in\stab j$ and $\delta_\alpha\in\stab0$, we have $t_jg=\gamma_j \eta_\alpha t_j\delta_\alpha$.  Hence
\[g\delta_\alpha^{-1}=t_j^{-1}\gamma_j\eta_\alpha t_j\in t_j^{-1}\stab jt_j\cap\stab 0,\]
so $g\in (t_j^{-1}\stab jt_j\cap\stab 0)\delta_\alpha.$
\end{proof}

The Hecke operator $T_s=\gl(2,\Z)s\gl(2,\Z)$ acts in the usual way on the homology $H_1(\gl(2,\Z),\MSQ)$ (see below), and hence, via the Shapiro isomorphism on the group 
\[\bigoplus_{\xS\in\mathcal A} H_1(\stab{\xS},\F).\]
We now work out the details of the action on this latter group. To do this, we use the following lemma concerning transfers and corestrictions. This lemma is standard, and follows easily from \cite[Sec. III.9]{Brown}.

\begin{lemma} \label{transfer}
Let $A$ be an infinite cyclic group with generator $a$, and $B\subset A$ a subgroup of index $d$, and suppose that $A$ acts trivially on $\F$.  For a group $G$ acting trivially on $\F$, we identify $H_1(G,\F)$ canonically with $G^{ab}\otimes_\Z\F$.
\begin{enumerate}[(i)]
\item The transfer map $\transfer:H_1(A,\F)\to H_1(B,\F)$ takes the generator $a\otimes 1$ to the generator $a^d\otimes 1$.
\item The corestriction map $\corestriction:H_1(B,\F)\to H_1(A,\F)$ (i.e. the map induced by the inclusion $B\subset A$) takes the generator $a^d\otimes 1$ to $d$ times the generator $a\otimes 1$.
\end{enumerate}
\end{lemma}

With notation as above, we will use the techniques of \cite{AD1} to write the Hecke operator $T_s$ (given as a sum of actions of all the $s_\alpha$) acting on a generator $z_0$,  as a sum of partial Hecke operators $U_{t_j}$ given as a sum of actions of the $s_{\beta,j}\in B_j$, so that $U_{t_j}$ maps $H_1(\stab0,\F)$ to $H_1(\stab j,\F)$.  

More precisely, if $F_\bullet$ is a resolution of $\F$ by free $\F\gl(2,\Q)$-modules, then $T_s$ on $H_1(\gl(2,\Z),\MSQ)$
sends the class of a cycle (using an obvious notation for elements of $F_\bullet$ and $\MSQ$) $\sum_f f\otimes_{\F\gl(2,\Z)} m(f)$ to  the class of
$\sum_\alpha \sum_f s_\alpha f\otimes_{\F\gl(2,\Z)} s_\alpha m(f)$.
The partial Hecke operator $U_{t_j}$ sends the class of a cycle  
$\sum_f f\otimes_{\stab0} \lambda(f)$ 
 to the class of  $\sum_{s_{\beta,j}\in B_j}\sum_f s_{\beta,j}f\otimes_{\stab j} \lambda(f)$.
 
 The following lemma follows immediately from Theorem 3.1 in \cite{AD1}.
 
 \begin{lemma}\label{ad1}
$T_s$ composed with the Shapiro isomorphism equals $\sum_{j=1}^J U_{t_j}$.
 \end{lemma}
 
 From now on we will also use  $T_s$ to stand for $\sum_{j=1}^J U_{t_j}$, depending on the context.

Next, we write the partial Hecke operator in terms of the transfer, corestriction, and an adjoint map.

\begin{theorem}\label{Thm:partial-hecke} Recall that $x_0\in\mathcal A$, and let $z_0$ be the generator of $H^1(\stab0,\F)$ chosen in Definition~\ref{Def:zx}.  Then $U_{t_j}:H_1(\stab0,\F)\to H_1(\stab j,\F)$ is given as the composition of the three maps
\[H_1(\stab0,\F)\to H_1(\stab0\cap s_j^{-1}\stab js_j,\F)
\xrightarrow{\phi_j}
H_1(s_j\stab0s_j^{-1}\cap \stab j,\F)\to H_1(\stab j,\F),\]
where the first map is the transfer, the second map is the map induced on homology by the pair of maps $(\Ad t_j,\tau_j)$, where $\Ad t_j$ is conjugation by $t_j$ on the group, and $\tau_j$ is multiplication by $q(\zeta_j)$ on the coefficient module, and the third map is corestriction.
\end{theorem}

\begin{proof}
We take a resolution $F_\bullet$ of $\F$ by free $\F[\gl(2,\Q)]$-modules and let $Z_0$ be a cycle representing $z_0$.  The map in Shapiro's lemma taking $H_1(\stab0,\F)$ into $H_1(\GL(2,\Z),\MSQ)$  sends $Z_0$ to $Z_0\otimes e_*(I\otimes 1)$.  Then, on the level of cycles, we have
\[T_s(Z_0\otimes e_*(I\otimes 1))=\sum_{j=1}^J\sum_{s_{\beta}\in B_j}s_{\beta}Z_0\otimes s_{\beta}e_*(I\otimes 1).\]
If we let $w_j$ be the composition of the three maps in the statement of the theorem, then using Lemma~\ref{ad1}  we see that what we need
 to show is that
\[w_j(Z_0)\otimes e_*(I\otimes 1)=\sum_{s_{\beta}\in B_j}s_{\beta}Z_0\otimes s_{\beta}e_*(I\otimes 1).\]

On the level of cycles, using Lemma~\ref{L:stab-reps}, the transfer of $Z_0$ is
\[\sum_{s_{\beta}\in B_j}\delta_{\beta}Z_0,\]
where we recall that for $s_\beta\in B_j$, we have written $s_\beta=\eta_\beta t_j\delta_\beta$, with $\eta_\beta\in\stab j$ and $\delta_\beta\in\stab0$.
Applying $\phi_j$ yields
\[q(\zeta_j)\sum_{s_\beta\in B_j}(t_j\delta_\beta t_j^{-1})(t_jZ_0).\]
Finally, applying the corestriction gives
\[q(\zeta_j)\sum_{s_\beta\in B_J}t_j\delta_\beta Z_0=q(\zeta_j)\sum_{s_\beta\in B_j}s_\beta Z_0.\]

We see that it suffices to show that $q(\zeta_j)e_*(I\otimes 1)=s_\beta e_*(I\otimes 1)$.

We have that
\begin{align*}
s_\beta e_*(I\otimes 1)&=s_\beta\sum_{\zeta_r\in\zreps}q(\zeta_r^{-1})\zeta_r\xS_0\cr
&=\sum_{\zeta_r\in\zreps}q(\zeta_r^{-1})\zeta_rs_\beta\xS_0\cr
&=\sum_{\zeta_r\in\zreps}q(\zeta_r^{-1})\zeta_r\zeta_j\xS_j\cr
&=\zeta_j\sum_{\zeta_r\in\zreps}q(\zeta_r^{-1})\zeta_r\xS_j\cr
&=q(\zeta_j)\sum_{\zeta_r\in\zreps}q(\zeta_r^{-1})\zeta_r\xS_j\cr
&=q(\zeta_j)e_*(I\otimes 1).\qedhere
\end{align*}
\end{proof}

We now apply Theorem~\ref{Thm:partial-hecke} to compute $U_{t_j}(z_0)$.

\begin{corollary}\label{C:partial-hecke}
The partial Hecke operator $U_{t_j}$ in Theorem~\ref{Thm:partial-hecke} satisfies
\[U_{t_j}(z_0)=e_jq(\zeta_j)z_j\]
where $e_j=[\stab j:t_j\stab0t_j^{-1}\cap\stab j]$.
\end{corollary}

\begin{proof}
By definition, $d_j=[\stab0:t_j^{-1}\stab jt_j\cap\stab 0]$ and $e_j=[\stab j:t_j\stab0t_j^{-1}\cap\stab j]$. Hence
 $g_0^{d_j}$ is a generator of $\stab0\cap t_j^{-1}\stab jt_j$, and
  $g_j^{d_j}$ is a generator of $t_j\stab0t_j^{-1}\cap\stab j$.  
  Considering $H_1(\stab0,\F)$ as $\stab0\otimes_\Z F$ we have $z_0=g_0\otimes 1$.  Hence, by Lemma~\ref{transfer}(i), the transfer takes $z_0$ to $g_0^{d_j}\otimes 1$, which is then mapped to $g_j^{d_j}\otimes q(\zeta_j)$ by $\phi_j$.  Finally, by Lemma~\ref{transfer}(ii), the corestriction maps this to $e_j(g_j\otimes q(\zeta_j))=e_jq(\zeta_j)(g_j\otimes 1)=e_jq(\zeta_j)z_j$. 
\end{proof}

We now compute the value of $e_j$.  Recall that $t_j\xS_0=\zeta_j\xS_j$.  
For $j=0,\dots,J$, choose $y_j\in Y$ such that $\xS_j$ is represented by $y_j$, and recall that $\stab j$ is the stabilizer of
 $\xS_j$ in $\gl(2,\Z)$ and $\epsilon$ is the fundamental unit of $\OK$ which we chose at the beginning of Section~\ref{S:Lattices}. Let $g_j\in\Gamma_j$ and $m_j\in\Z$ be defined as in Definition~\ref{D:generator} (with $H=\KSQ$).  Then $g_j$ is a generator of $\Gamma_j$ and $m_j>0$.  Set $\delta_j=\pm\epsilon^{m_j}$, where the sign is chosen so that $g_jy_j=y_j\delta_j$ and $\delta_j\in\KSQ$.
 
\begin{lemma}\label{H5} With notation as above, 
\[e_j=\lcm(m_0,m_j)/m_j,\qquad d_j=\lcm(m_0,m_j)/m_0,\]
 and $e_jm_j=d_jm_0$.
\end{lemma}
\begin{proof}
First, 
$g_0y_0=y_0\delta_0.$  Since $t_j\xS_0=\zeta_j\xS_j$,  we have $t_jy_0=\alpha_j\zeta_jy_j$ for some $\alpha_j\in\KSQ$.  Hence, $t_j^{-1}y_j=\alpha_j^{-1}\zeta_j^{-1}y_0$.  

It follows  that $t_jg_0t_j^{-1}y_j=y_j\delta_0$.  In addition, $g_jy_j=y_j\delta_j$, 
$g_0$ generates $\Gamma_0$,  and $g_j$ generates $\Gamma_j$.  

We may choose a generator $h$ of $\Gamma_j\cap t_j\Gamma_0t_j^{-1}$, so that $h$ will be the smallest power of $t_jg_0t_j^{-1}$ that is contained in $\Gamma_j$.  This power must be the smallest positive integer $k$ such that $\delta_0^k$ is a power of $\delta_j$.  Since $\delta_0,\delta_j\in\KSQ$ and $-1\notin\KSQ$, we see that $k$ will be the smallest positive integer such that $km_0$ is a multiple of $m_j$.  Hence, $km_0=\lcm(m_0,m_j)$, and we see that
\[h y_j=y_j\left(\pm\epsilon^{\lcm(m_0,m_j)}\right).\]
It follows that 
\[e_j=\lcm(m_0,m_j)/m_j.\]
Reversing the roles of $\Gamma_0$ and $\Gamma_j$ and switching $t_j$ and $t_j^{-1}$, we obtain
\[d_j=\lcm(m_0,m_j)/m_0.\qedhere\]
\end{proof}

\section{Elements of $H^1(\gl(2,\Z), \MSQ^*)$ interpreted as functions on lattices}

We now interpret the cohomology of the dual of $\MSQ$ as a collection of functions on a space of lattices.

\begin{definition}
Let $\Phi$ be a function from lattices in $K$ to $\F$.  We will say that $\Phi$ is $q$-homogeneous if $\Phi(\alpha L)=q(\alpha I)\Phi(L)$ for all $\alpha\in \Q^\times\cap K_{S_0}$ and all lattices in $L$.

We will say that $\Phi$ is $\KSQ$-invariant if $\Phi(\alpha L)=\Phi(L)$ for all $\alpha\in\KSQ$ and all lattices $L$.
\end{definition}

\begin{remark} Note that since $q$ is trivial on $\KSQ$, a function $\Phi$ can be both $q$-homogeneous and $\KSQ$-invariant.  In addition, we note that since $K$ is a real quadratic field, $q(-I)=1$.  If this were not the case, the fact that $-L=L$ for any lattice $L$ in $K$ would force all $q$-homogeneous functions to be identically 0.
\end{remark}

\begin{lemma}\label{Lemma:homothety-to-cohomology}
There is an isomorphism between $H^1(\gl(2,\Z),\MSQ^*)$ and the vector space of $\F$-valued functions on lattices in $K$ that are $q$-homogeneous and $\KSQ$-invariant.
\end{lemma}

\begin{proof} 
Choose a set $\A$ of representatives of the $Z\,\gl(2,\Z)$-orbits in $\XSQ$. This choice of $\A$ yields an isomorphism of $\gl(2,\Z)$-modules
\[f:\MSQ\to\bigoplus_{\xS\in\mathcal A}\F\,\gl(2,\Z)\otimes_{\F\stab{\xS}}\F.\]

This induces an isomorphism (via Shapiro's Lemma)
\begin{align*}
H_1(\gl(2,\Z),\MSQ)&\cong\bigoplus_{\xS\in\mathcal A}H_1(\gl(2,\Z),\F\,\gl(2,\Z)\otimes_{\F\stab{\xS}}\F)\cr
&\cong \bigoplus_{\xS\in\mathcal A}H_1(\stab{\xS},\F)\cr
&\cong\bigoplus_{\xS\in\mathcal A}\F.
\end{align*}
Using the natural duality between $H_1(\gl(2,\Z),\MSQ)$ and $H^1(\gl(2,\Z),\MSQ^*)$, we see that determining an element of $H^1(\gl(2,\Z),\MSQ^*)$ is the same as giving a function from $\mathcal A$ to $\F$.  

We now show that there is an isomorphism between the vector space of functions from $\mathcal A$ to $\F$ and the vector space of $\KSQ$-invariant $q$-homogeneous functions on lattices in $K$.

Let $h$ be any $q$-homogeneous $\KSQ$-invariant function on lattices in $K$.  Since every element in $\A$ can be lifted uniquely to a $\KSQ$-homothety class of lattices in $K$, $h$ defines a function $g$ on $\mathcal A$.  
Namely, given $x\in\mathcal A$, lift $x$ to $y\in Y$ and set $g(x)=h(\Lambda_y)$, where $\Lambda_y$ is the lattice spanned by the entries of $y$.  Since $y$ is well-defined up to $\KSQ$-homotheties and $h$ is $\KSQ$-invariant, this gives a well-defined function $g$.

Given a function $g$ on $\mathcal A$ and a lattice $L$ in $K$, $L$ corresponds (by choosing a basis $y=\transpose{(a,b)}\in Y$) to an element $x'\in \XSQ$, which lies in the $Z\,\gl(2,\Z)$-orbit of a unique $x\in\mathcal A$.  Let $x'=\zeta\gamma x$, with $\zeta\in Z$ and $\gamma\in\gl(2,\Z)$.  Define $h(L)=q(\zeta)g(x)$. Note that if $x'=\zeta'\gamma'x$ with $\zeta'\in Z$ and $\gamma'\in\gl(2,\Z)$, then we have $(\zeta^{-1}\zeta')(\gamma^{-1}\gamma')x=x$.  Hence, $(\zeta^{-1}\zeta')(\gamma^{-1}\gamma')=r_y(\alpha)$ for some $\alpha\in \KSQ$.  By Lemma~\ref{Lemma:Z-KSQ}, this implies that $q(\zeta)=q(\zeta')$, so $h$ is well defined.   Then $h$ is a $q$-homogeneous $\KSQ$-invariant function on lattices.

These two maps (taking $h$ to $g$ and $g$ to $h$) are easily seen to be inverses, and preserve addition and scalar multiplication.
\end{proof}

\section{The Branched Bruhat-Tits graph and the Laplacian}

In order to construct functions on lattices that are eigenfunctions of the Hecke operators, we will use a modification of the Bruhat-Tits building \cite{Casselman,Serre}, in which we lift the Bruhat-Tits building to a finite branched cover.

For each prime $\ell$ unramified in $K$, let $K_\ell$ denote $K\otimes\Q_\ell$.  Then $K_\ell$ is a two-dimensional vector space over $\Q_\ell$.

\begin{definition}\label{identify}
If $\ell$ is inert, then $K_\ell$ is a quadratic field extension of $\Q_\ell$.  We fix the integral basis $\{1,\omega\}$ of $K$, and we identify $K_\ell$ with $\Q_\ell^2$ by identifying $1$ and $\omega$ with the standard basis elements $e_1,e_2\in\Q_\ell^2$.

If $\ell$ splits in $K$, then $(\ell)=\lambda\lambda'$ for prime ideals $\lambda,\lambda'$ in $\OK$ lying over $\ell$.  Each of the completions $K_\lambda$ and $K_\lambda'$ is then isomorphic to $\Q_\ell$.  Restricting these isomorphisms to $K$, we obtain two distinct Galois conjugate embeddings $i_\lambda,i_{\lambda'}:K\to\Q_\ell$.  We then identify $K_\ell=K\otimes_\Q\Q_\ell$ with $\Q_\ell^2$ via the map taking
\[t\otimes 1\mapsto (i_\lambda(t),i_{\lambda'}(t)).\]  We abbreviate the notation by writing $t\mapsto (t,t')$.
\end{definition}

\begin{definition} 
By a lattice in $K_\ell$, we will mean a rank two $\Z_\ell$-submodule of $K_\ell$.
\end{definition}

If $L$ is a lattice in $K$, then $L_\ell=L\otimes_\Z\Z_\ell$ is a lattice in $K_\ell$.

\begin{definition} Let $\ell$ be a prime, and $n$ a positive integer.  Denote the elements of $\Q_\ell^\times$ with $\ell$-adic valuation divisible by $n$ by $V_n$.  We note that $V_n$ is a subgroup of index $n$ of $\Q_\ell^\times$.
\end{definition}

\begin{definition} Let $L_1$ and $L_2$ be lattices in $K_\ell$.  We say that $L_1$ and $L_2$ are {\em $n$-homothetic} if $L_1=\alpha L_2$ for some $\alpha\in V_n$.  Then $n$-homothety is an equivalence relation, and we call an equivalence class an $n$-homothety class of lattices in $K_\ell$.
\end{definition}

\begin{definition} Let $n$ a positive integer, $K$ a real quadratic field, and  $\ell$ a prime unramified in $K$ .  The branched Bruhat-Tits graph $\T_\ell^n$ is the graph whose vertices are $n$-homothety classes of lattices in $K_\ell$.  Two vertices are joined by an edge if there are representative lattices $L_1$ and $L_2$ of the vertices such that $L_2\subset L_1$ or $L_1\subset L_2$ with index $\ell$.
\end{definition}

\begin{remark} The Bruhat-Tits tree is a special case of the branched Bruhat-Tits graph in which $n=1$.  When $n=1$, we may denote $\T_\ell^n$ by $\T_\ell$. When $n>1$, we will typically write vertices of $\T_\ell^n$ with a superscript $n$, i.e.{} $t^n\in\T_\ell^n$.
\end{remark}

\begin{definition} Let $L$ be a lattice in $K_\ell$.  Denote the vertex of $\T_\ell^n$ represented by $L$ by $\varpi(L)$.  Denote the vertex of $\T_\ell$ represented by $L$ by $\pi(L)$.  Given a vertex $t^n\in\T_\ell^n$, there is a unique vertex $s\in\T_\ell$ containing $t^n$; we write $s=\pi(t^n)$. 
\end{definition}

Note that for any lattice $L$ with $\varpi(L)=t^n$, $\pi(t^n)=\pi(L)$.  To keep our notation less cluttered, 
if $L$ is a lattice in $K_\ell$, we will often denote $\varpi(L)$ by $L$,
 as long as the context makes this usage clear.

\begin{remark} We note that for any vertex $t\in\T_\ell$, there are exactly $n$ vertices $t^n\in \T_\ell^n$ with $\pi(t^n)=t$.  If $L$ is a lattice in $K_\ell$ representing $t$, these $n$ vertices of $\T_\ell^n$ are represented by
\[L,\ell L,\ldots,\ell^{n-1}L.\]
\end{remark}

\begin{definition} If $t$ is a vertex of $\T_\ell$, we will call the set $\{t^n\in\T_\ell^n: \pi(t^n)=t\}$ the {\it fiber} of $t$ and also the {\it fiber} of $t^n$ for any $t^n$ in that set .\end{definition}

\begin{definition}\label{D:idealistic}
A vertex $t^n$ is {\em idealistic} if $t^n$ is the $n$-homothety class of $I_\ell$ for some fractional ideal $I$ of $K$.  
\end{definition}

We now review some facts about completions $L_\ell$ of lattices in $K$.  Let $\ell$ be a prime of $\Q$.  

By \cite[V.2, Corollary to Theorem 2]{Weil}, the operations of sum and intersection of lattices in $K$ commute with completion at $\ell$.  In addition, by \cite[V.3, Theorem 2]{Weil}, a lattice $L$ in $K$ is determined by its set of completions $L_w$ for all finite places $w$ of $\Q$.  In fact
\[L=\bigcap_w K\cap L_w.\]
Finally, completion at a finite place $w$ of finitely generated $\Z$-modules is an exact functor \cite[Theorem 7.2]{Eisenbud}.

Applying these facts to fractional ideals of $K$, we note that if $I$ is an ideal of $\OK$ of norm prime to $\ell$, then $I_\ell$ is an ideal of $\OK_\ell$ of index prime to $\ell$, so $I_\ell=\Ol$.  In addition, multiplication of relatively prime ideals (i.e.~intersection) commutes with completion at $\ell$.  Hence, for an ideal $I$, the completion $I_\ell$ depends only on the factors of $I$ of $\ell$-power norm. 

Now suppose that $t^n\in\T_\ell^n$ is idealistic.  Then we may assume that $t^n$ is represented by an ideal $I_\ell$, where $I$ is an ideal with $\ell$-power norm in $\OK$.  If $\ell$ is inert in $K$, such an $I$ must be principal, so $I_\ell$ is $\Q_\ell$-homothetic to $\Ol$.  Hence, $t^n$ is idealistic if and only if $\pi(t^n)$ is represented by $\Ol$.

On the other hand, if $\ell$ splits in $K$, then $\ell\OK=\lambda\lambda'$, where $\lambda,\lambda'$ are prime ideals of $\OK$ lying over $\ell$.   We then see that  $t^n\in\T_\ell^n$ is idealistic if and only if $\pi(t^n)$ is represented by an ideal of the form $\lambda^k$ or $(\lambda')^k$.  In particular, if $\pi(t_1^n)=\pi(t_2^n)$, then $t_1^n$ and $t_2^n$ are either both idealistic, or both nonidealistic.

\begin{lemma} Let $L_1\supset L_2$ be lattices in $K_\ell$ with $[L_1:L_2]=\ell$.  Let $t_1^n=\varpi(L_1)\in\T_\ell^n$.  Then there are precisely two vertices $t_2^n,t_3^n\in\T_\ell^n$ with $\pi(t_2^n)=\pi(t_3^n)=\pi(L_2)$, such that there is an edge between the two pairs $(t_1^n,t_2^n)$ and $(t_1^n,t_3^n)$.  If we let $t_2^n$ be represented by $L_2$, then $t_3^n$ is represented by $\ell^{-1} L_2$.
\end{lemma}

\begin{proof}
Clearly, if we take $t_2^n=\varpi(L_2)$ and $t_3^n=\varpi(\ell^{-1}L_2)$, we see that $t_2^n$ and $t_3^n$ are distinct and have the desired properties.  It remains to show that there is no third vertex $t_4^n$, distinct from $t_2^n$ and $t_3^n$, with $\pi(t_4^n)=\pi(L_2)$, and such that there is an edge between $t_4^n$ and $t_1^n$.  

Suppose that there is an edge between $t_1^n$ and $t_4^n$.  Then either there is a lattice $L_3$ representing $t_4^n$ such that $L_1\supset L_3$ and $[L_1:L_3]=\ell$ or there is a lattice $L_3$ representing $t_4^n$ such that $L_1\subset L_3$ and $[L_3:L_1]=\ell$.  

Now suppose $L_3$ is homothetic to $L_2$, say with $L_3=\alpha L_2$, where $\alpha\in\Q_\ell^\times$.  

If $L_1\supset L_3$ has index $\ell$, then we have that $\ell^{-1}L_2\supset L_1$ has index $\ell$ and $L_1\supset\alpha L_2$ has index $\ell$.  Hence, multiplying by $\ell$, we see that $L_2\subset \ell\alpha L_2$ with index $\ell^2$.  This implies that $v_\ell(\alpha)=0$, so that $\alpha L_2=L_2$, so $t_4^n=t_2^n$.

On the other hand, if $L_1\subset L_3$ with index $\ell$, we have that $L_2\subset\alpha L_2$ has index $\ell^2$.  Hence $v_\ell(\alpha)=-1$, and we see that $\alpha L_2=\ell^{-1}L_2$, so $t_4^n=t_3^n$. 
\end{proof}

\begin{corollary} Let $n\geq 1$, let $t^n$ be a vertex in $\T^n_\ell$, and let $t=\pi(t^n)\in\T_\ell$.  Let $s\in\T_\ell$ be a neighbor of $t$.  Then there are exactly two neighbors $s_1^n$ and $s_2^n$ of $t^n$ in $\T^n_\ell$ with $\pi(s_1^n)=\pi(s_2^n)=s$.  If $L$ represents $t^n$, then exactly one of $s_1^n$ and $s_2^n$ is represented by a sublattice $L'$ of $L$ of index $\ell$; the other is represented by $\ell^{-1}L'$, which contains $L$ with index $\ell$.
\end{corollary}

\begin{definition} Let $n\geq 1$, let $t^n\in\T_\ell^n$ be a vertex represented by a lattice $L$ in $K_\ell$, and let $s_1^n, s_2^n\in\T_\ell^n$ be two neighbors of $t_n$ with $\pi(s_1^n)=\pi(s_2^n)$.  Call the neighbor represented by a sublattice of index $\ell$ in $L$ a {\it downhill} neighbor of $t^n$; call the other an {\em uphill} neighbor of $t$.
\end{definition}

\begin{definition} Let $t^n\in\T_\ell^n$.  We define the {\em tier} of $t^n$ to be the distance between $\pi(t^n)$ and $\pi(\Ol)$ in $\T_\ell$.  A neighbor of $t^n$ of higher tier than $t^n$ will be called an {\em outer neighbor} of $t^n$: a neighbor of lower tier will be called an {\em inner neighbor}.
\end{definition}

\begin{remark} Each $t^n\in\T_\ell^n$ has precisely $\ell+1$ downhill neighbors and $\ell+1$ uphill neighbors.  The use of uphill and downhill matches our intuition; if $s^n$ is a downhill neighbor of $t^n$, then $t^n$ is an uphill neighbor of $s^n$.

Each vertex of positive tier has precisely $\ell$ downhill outer neighbors, and $1$ downhill inner neighbor.  It also has precisely $\ell$ uphill outer neighbors, and $1$ uphill inner neighbor.

A vertex of tier 0 has only outer neighbors; $\ell+1$ of them are uphill, and $\ell+1$ are downhill.
\end{remark}

There is a natural action of the group $\GL(2,\Q_\ell)$ on $\Q_\ell^2$, namely matrix multiplication with elements of $\Q_\ell^2$ considered as column vectors.  We transfer this action to $K_\ell$ via the identification that we have made between $K_\ell$ and $\Q_\ell^2$.  The action of  $g\in\gl(2,\Q_\ell)$ is invertible, and preserves $\Q_\ell$-linear  combinations, so it maps bases of $\Q_l^2$ to bases, maps lattices to lattices, and preserves $n$-homothety of lattices.  Hence, multiplication by $g$ defines a bijection from $\T_\ell^n$ to $\T_\ell^n$.

\begin{definition} 
Let $F(\T_\ell^n)$ be the set of $\F$-valued functions on the vertices of $\T_\ell^n$.
\end{definition}

\begin{definition} The {\em Laplace operator} $\Delta_\ell^n$ on $F(\T_\ell^n)$ is defined by
\[\Delta_\ell^n(f)(t^n)=\sum_{u^n}f(u^n),\]
where the sum runs over the $\ell+1$ downhill neighbors $u^n\in\T_\ell^n$ of $t^n\in \T_\ell^n$.
\end{definition}
  
In the next lemma, we describe how the coset representatives for a Hecke operator act on lattices.  Recall Lemmas~\ref{L:partition}, \ref{L:stab-reps} and \ref{ad1} for the definition of the sets $B_j$, the coset representatives $s_{\beta,j}$, and the integers $d_j$.  Note that these definitions depend on a choice of an element $x_0\in\XSQ$ and a choice of $Z\,\gl(2,\Z)$-orbit representatives $\A$ containing $x_0$.

\begin{lemma}\label{ell+1}
Let $x_0\in\XSQ$ and let $\A$ be a set of $Z\,\gl(2,\Z)$-orbit representatives containing $x_0$. Let $x_0$ be represented by $y=(a_0,b_0)\in Y$ with $a_0,b_0\in\OK$, and let $L_0=L(a_0,b_0)$ be the lattice generated by $a_0$ and $b_0$. Let $s=\diag(\ell,1)$ and let 
\[\gl(2,\Z)s\gl(2,\Z)=\coprod_{\alpha}\gl(2,\Z)s_\alpha\]
with the $s_\alpha$ chosen and partitioned as described in Section~\ref{coefficients}.
\begin{enumerate}[(i)]
\item $\mathcal L=\{s_\alpha L_0\}$ consists of the $\ell+1$ lattices of index $\ell$ contained in $L_0$.
\item $\mathcal L$ is partitioned into the subsets
$$\mathcal L_j=\{s_{\alpha}L_0|s_\alpha\in B_j\},$$
and $|\mathcal L_j|=d_j$.
\item The same is true of the completions at $\ell$: $\mathcal L_\ell=\{s_\alpha( L_0)_\ell\}$ consists of the $\ell+1$ lattices of index $\ell$ contained in $(L_0)_\ell$, and these are partitioned into the subsets 
$$\mathcal L_{\ell,j}=\{s_{\alpha}(L_0)_\ell|\alpha\in B_j\}.$$
and $|\mathcal L_{\ell,j}|=d_j$
\end{enumerate}
\end{lemma}
\begin{proof}
\begin{enumerate}[(i)]
\item If 
$$s_\alpha\begin{pmatrix}a_0\cr b_0\end{pmatrix}=\begin{pmatrix}a_\alpha\cr b_\alpha\end{pmatrix},$$
then $s_\alpha L_0=L(a_\alpha,b_\alpha)$.  Since $s_\alpha$ is an integral matrix of determinant $\ell$, it is clear that $L(a_\alpha,b_\alpha)$ has index $\ell$ in $L_0$, and all sublattices of $L_0$ of index $\ell$ arise this way.
\item Since $\{s_\alpha\}$ is partitioned by the sets $B_j$, it is clear that the lattices are partitioned as indicated.
\item If $L$ has index $\ell$ in $L_0$, then the completion $L_\ell$ has index $\ell$ in $(L_0)_\ell$, since taking completions of finitely generated modules is an exact functor. Given two lattices $L\neq M$, each having index $\ell$ in $L_0$, we note that for all places $w\neq\ell$, $L_w=M_w=(L_0)_w$.  Since a lattice is determined by its completions at all finite places, we must have $L_\ell\neq M_\ell$.
\end{enumerate}
\end{proof}


\begin{definition}
Let $\phi_\ell^n\in F(\T_\ell^n)$, let $x_0\in\XSQ$, and let $\A$ be any set of $Z\,\gl(2,\Z)$-orbit representatives of $\XSQ$ containing $x_0$.  Define the sets $B_j$ in terms of $x_0$ and $\A$ as in Lemma~\ref{ell+1}. If, for all choices of $\A$ and for all $y=\transpose{(a_0,b_0)}\in Y$ with $a_0,b_0\in\OK$ representing $x_0$, we have that $\phi_\ell^n$ is constant on the set
 \[\{(s_{\beta,j}L(a_0,b_0))_\ell|\beta=1,\ldots,d_j\}\]
of vertices of $\T_\ell^n$, then we will say that $\phi_\ell^n$ is {\em locally constant} relative to $T_\ell$ and $x_0$.

If $\phi_\ell^n$ is locally constant relative to $T_\ell$ and all $x_0\in\XSQ$, then we say that $\phi_\ell^n$ is {\em locally constant}.
\end{definition}

We remark that the condition $a_0,b_0\in\OK$ could be relaxed to $a_0,b_0\in K$ without effect.  This is true because for any pair $a_0,b_0\in K$, there is an integer $m$ such that $m^na_0,m^nb_0\in\OK$; then $L(a_0,b_0)$ and $L(m^na_0,m^nb_0)$ are $n$-homothetic and hence define the same vertex of $\T_\ell^n$.

\begin{definition} Let $t_0^n\in\T_\ell^n$ be the vertex represented by the lattice $\Ol$.
\end{definition}

\begin{lemma}\label{L:gl2zl-invariant}
The action of $\gl(2,\Z_\ell)$ on $\T_\ell^n$ permutes the vertices of $\T_\ell^n$, fixes vertices of tier 0, and preserves edges (including whether the edge is uphill or downhill) and the tier of each vertex.
\end{lemma}

\begin{proof}
Since the action of $\GL(2,\Z_\ell)$ is invertible, it is clear that the map it induces on vertices is a bijection.  In addition, if $\gamma\in\gl(2,\Z_\ell)$, and $L_1\subset L_2$ with index $\ell$, then $\gamma L_1\subset \gamma L_2$ with index $\ell$, so edges are preserved (including whether the edge is uphill or downhill).

Since the action of $\gl(2,\Z_\ell)$ fixes $\Z_\ell^2$, which is identified with $\Ol$, it fixes vertices of tier 0.  Since it preserves neighbors, a simple inductive argument shows that it maps each vertex to a vertex of the same tier.
\end{proof}

\begin{lemma}\label{mult-epsilon}
Multiplication by the fundamental unit $\epsilon\in K\subset K_\ell$ induces a permutation on the vertices of $\T_\ell^n$  given (on the level of $\Z_\ell$-lattices) by multiplication by a matrix in $\GL(2,\Z_\ell)$. 
\end{lemma}
\begin{proof}

Suppose that $\ell$ is inert in $K$.  In this case (see Definition~\ref{identify}), we have identified $K_\ell=\Q_\ell\oplus\Q_\ell\omega$ with $\Q_\ell^2$.  Since multiplication by $\epsilon$ is $\Q$-linear on $K$ it induces a $\Q_\ell$-linear map on $K_\ell$. Hence, multiplication by $\epsilon$ is represented by a matrix in $\GL(2,\Q_\ell)$.  Since multiplication by $\epsilon$ is an automorphism of $\Ol$, and $\Ol$ is identified with $\Z_\ell^2\subset\Q_\ell^2$, this matrix has entries in $\Z_\ell$, and since $\epsilon$ has norm $\pm1$, the matrix must have determinant $\pm1$, and we see that the matrix is in $\GL(2,\Z_\ell)$.

Now suppose that $\ell$ is split.  Referring to Definition~\ref{identify} again,  we have identified $K_\ell$ with $\Q_\ell^2$, where $c\in K$ is identified with $(c,c')\in\Q_\ell^2$.  Hence, multiplication by $\epsilon$ is represented by the matrix $\diag(\epsilon,\epsilon')$, which is in $\GL(2,\Z_\ell)$.  
\end{proof}


\begin{lemma}\label{loc-const} Let $\phi\in F(\T^n_\ell)$ be a function on the vertices of $\T^n_\ell$.  Assume that for every vertex $t^n\in\T^n_\ell$, $\phi$ is constant on the set of non-idealistic outer downhill neighbors $u^n$ of $t^n$.  Then $\phi$ is locally constant relative to $T_\ell$ and any $x_0\in\XSQ$.
\end{lemma}

\begin{proof}
Assume that $\phi$ satisfies the conditions of the lemma.  Let $x_0\in\XSQ$, choose any collection $\A$ of orbit representatives containing $x_0$, and choose any $y=\transpose{(a_0,b_0)}$ with $a_0,b_0\in\OK$ representing $x_0$.  Partition the set $\{s_\alpha\}$ of coset representatives for the Hecke operator $T_\ell$ as in Lemma~\ref{ell+1}.

Let $t^n$ be the vertex of $\T_\ell^n$ represented by $L_0=L(a_0,b_0)$.  For each set $B_j$, we wish to show that $\phi$ is constant on the set $\{(s_{\beta,j}L_0)_\ell|s_{\beta,j}\in B_j\}$.  Choose any $s_{\beta,j}$ and let $u^n$ be the downhill neighbor of $t^n$ represented by $(L_1)_\ell$, where $L_1=s_{\beta,j}L_0$.  Then $L_1$ is homothetic to a lattice with a basis representing $x_j$.  We now divide the proof into 3 cases.
\begin{enumerate}[{Case }1.,leftmargin=*]
\item Suppose $u^n$ is idealistic.  Then $L_1$ is a fractional ideal of $K$, and is homothetic to a fractional ideal with basis representing $x_j\in\A$.  Hence, $m_j=i^S$, so $d_j=1$ by Lemmas~\ref{is} and~\ref{H5}.  Hence, there is only one vertex on which $\phi$ must be constant.
\item Suppose that $u^n$ is the unique downhill inner neighbor of $t^n$.  Recall from Theorem~\ref{stabgen} that $\stab{0}=\Gamma{x_0}$ fixes $x_0$ and is generated by an element $g_0$ that acts on $L_0$ as multiplication by $\delta_0=\pm\epsilon^{m_x}$ with the sign chosen so that $\delta_0\in \KSQ$.  From Theorem~\ref{stabgen}, we see that
\[g_0L_0=\delta_0L_0=L_0.\]
Since multiplication by $\delta_0$ fixes $L_0$, it also fixes $(L_0)_\ell=t^n$.  Multiplication by $\delta_0$ also fixes each element of the fiber of $t_0^n=\varpi(\Ol)$, so it must fix the unique downhill path from $t^n$ to the fiber of $t_0^n$.  Hence, multiplication by $\delta_0$ must fix $u^n$.  

Now, both $L_1$ and $\delta_0L_1$ are sublattices of $L_0$ of index $\ell$.  Since both must represent $u^n$, we see that they are equal.  Since $\delta_0L_1=L_1$, we see that $m_j|m_0$, so that $d_j=1$.  Hence, again, there is only one vertex on which $\phi$ must be constant.
\item Suppose $u^n$ is a nonidealistic outer downhill neighbor of $t^n$.  By cases 1 and 2, no vertex in $\{(s_{\beta,j}L_0)_\ell|\beta\in B_j\}$ can be idealistic or a downhill inner neighbor of $t_n$.  Hence, $\phi$ is constant (by hypothesis) on all the vertices in the desired set.
\end{enumerate}
\end{proof}

\section{Construction of functions on lattices; comparison between the Laplacian and a Hecke operator}

\begin{definition}\label{D:q-homogeneous} Let $q:Z\to\F^\times$ be the character defined in Definition~\ref{D:q*}, and let $\ell$ be a prime of $\Z$ that is unramified in $K$.   We say that a function $f\in F(\T_\ell^n)$ is $q$-homogeneous (or just homogeneous, if $q$ is understood) if, for all lattices $L$ in $K$,
\[f(\ell L_\ell)=q(\ell I)f(L_\ell).\]
\end{definition}

\begin{definition}\label{Def:glob} For all finite places $w$ of $\Q$ unramified in $K$, let $n_w=1$ if $w$ is inert in $K$, let $n_w=2$ if $w$ splits in $K$.  Fix a prime $\ell$ of $\Q$ not dividing $pdN$, and let $W$ be the set of all finite places of $\Q$ not dividing $\ell pdN$.
For $w\in W$, let $\phi_w\in F(\T_w^{n_w})$ denote a homogeneous function such that $\phi_w(\Ow)=1$.  We view the functions $\phi_w$ as fixed by the context, and do not include them in the following notation for $\Phi$.  For any homogeneous $\phi_\ell\in F(\T_\ell^{n_\ell})$, define the function $\Phi(\phi_\ell)$ on lattices $L$ in $K$ by the formula
$$\Phi(\phi_\ell)(L)=\phi_\ell(L_\ell)\prod_{w\in W}\phi_w(L_w).$$

\end{definition}
\begin{lemma}\label{new}
The infinite product in the definition makes sense and $\Phi(\phi_\ell)$ is $q$-homogeneous.
The map $\phi_\ell\mapsto \Phi(\phi_\ell)$ is $\F$-linear.
\end{lemma}
\begin{proof}
For any given $L$, we have that $L_w=\Ow$ for almost all $w$, so the product is actually finite.  The linearity of the map $\phi_\ell\mapsto \Phi(\phi_\ell)$ is clear.  Now suppose $\alpha\in \Q^\times\cap K_{S_0}$ and $L$ is a lattice.
Then $\alpha$ is prime to $pdN$ and factors as 
\[
\alpha=\ell^{f_\ell}\prod_{w\in W} w^{f_w}.
\]
  Then
\[\Phi(\phi_\ell)(\alpha L)=\phi_\ell(\alpha L_\ell)\prod_{w\in W}\phi_w(\alpha L_w)=
\phi_\ell(\ell^{f_\ell} L_\ell)\prod_{w\in W}\phi_w(w^{f_w} L_w)
.\]
Since $\phi_\ell$ and all the $\phi_w$ are homogeneous, this equals 
\[
q(\ell^{f_\ell}I)\phi_\ell( L_\ell)\left(\prod_{w\in W}q(w^{f_w}I)\right)\left(\prod_{w\in W}\phi_w(L_w)\right)=
q(\alpha I)\Phi(\phi_\ell)(L).\qedhere
\]
\end{proof}

We now proceed to the main theorem of this section: the comparison between the Hecke operator and the Laplace operator.

By Lemma~\ref{Lemma:mL} and the fact that $\KSQ$ is unit-cofinite,
we see that for any lattice $L\subseteq K$, there is a minimal positive integer $m_L$ such that $\epsilon^{m_L}L=L$ and one of $\pm\epsilon^{m_L}\in\KSQ$.    If $L_1$ and $L_2$ are $K^\times$-homothetic lattices in $K$, it is clear that $m_{L_1}=m_{L_2}$.  
 Set $m'_L=m_L/i^S$. By Theorem~\ref{stabgen}, if $L=L(a,b)$ and $x$ is the image in $\XSQ$ of $y={}^t(a,b)$ then $m_L=m_x$.  Therefore, by Lemma~\ref{is}, $i^S|m_L$ and $m'_L$ is a positive integer.

\begin{definition} Let $\psi_\ell\in F(\T_\ell^{n_\ell})$.  We define the {\em transform} of $\psi_\ell$ to be the function $\hat\psi_\ell\in F(\T_\ell^n)$ given by the formula
\[\hat\psi_\ell(t^n)=m'_L\psi_\ell(t^n),\]
where $L$ is any lattice in $\OK$ of $\ell$-power index, such that $L_\ell$ represents $t^n$.
\end{definition}

\begin{lemma} 
Given $\psi_\ell\in F(\T_\ell^n)$, the transform $\hat\psi_\ell$ is well defined.
\end{lemma}

\begin{proof} We need to show that for $t^n\in\T_\ell^n$, the value of $m_L$ does not depend on the lattice $L$ chosen to represent $t^n$.  Note that up to homothety by powers of $\ell^n$, there is a unique lattice $\Lambda\subseteq K_\ell$ representing $t^n$.  By \cite[V.2, Theorem 2]{Weil} there is a unique lattice $L\subseteq \OK$ of $\ell$-power index such that $L_\ell=\Lambda$.  Since $\Lambda$ is uniquely defined up to homothety by powers of $\ell^n$, so too is $L$.  Finally, since homothety does not change the value of $m_L$, we see that $m_L$ does not depend on the choice of $L$, so $m'_L$ does not.
\end{proof}

If $\psi_\ell(\Ol)=1$, then $\hat\psi_\ell(\Ol)=1$, since $m'_{\OK}=1$.  In addition, if $\F$ has characteristic $0$, then a function $\psi_\ell$ is determined by its transform; this fails if any $m'_L$ is divisible by the characteristic of $\F$.

\begin{lemma}
 Let 
$\ell\nmid pdN$ be prime.  If $\psi_\ell\in F(\T_\ell^n)$ is homogeneous, then $\hat\psi_\ell$ is also homogeneous.
\end{lemma}

\begin{proof}
If $t^n\in\T_\ell^n$ is represented by $L_\ell$, with $L$ a lattice of $\ell$-power index in $\OK$, then $\ell t^n$ is represented by $\ell L_\ell$.  Since $m_{L}'=m_{\ell L}'$, we have
\[\hat\psi_\ell(\ell t^n)=m_{\ell L}'\psi_\ell(\ell t^n)=m_L'q(\ell I)\psi_\ell(t^n)=q(\ell I)\hat\psi_\ell(t^n).\qedhere\]
\end{proof}

We now fix a set $\A_0$ of representatives of the $Z\,\gl(2,\Z)$-orbits in $\XSQ$.  Recall from Lemma~\ref{Lemma:homothety-to-cohomology} that this choice fixes an isomorphism between the cohomology group
\[H^1(\gl(2,\Z),\MSQ^*)\]
and $q$-homogenous, $\KSQ$-invariant functions on lattices.

\begin{theorem}\label{Theorem:hecke-transform}
Let $\ell\nmid pdN$ be prime.
For each finite place $w\in W$, fix a homogeneous function $\phi_w\in F(\T_\ell^{n_w})$, as in Definition~\ref{Def:glob}. Let $n=n_\ell$, and let $\psi_\ell\in F(\T_\ell^n)$ be homogeneous.  Assume that $\Phi(\hat\psi_\ell)$ is $\KSQ$-homothety invariant.  (It will be $q$-homogeneous by Lemma~\ref{new}.)

As in Lemma~\ref{Lemma:homothety-to-cohomology} and its proof, view $\Phi(\hat\psi_\ell)$ as an element of 
\[
H^1(\GL(2,\Z),\MSQ^*)\cong H_1(\GL(2,\Z),\MSQ)^*.
\]  That is to say, view $\Phi(\hat\psi_\ell)$ as an $\F$-valued functional on $H_1(\GL(2,\Z),\MSQ)$, via the pairing
\[\langle\Phi(\hat\psi_\ell),\bullet\rangle:H_1(\GL(2,\Z),\MSQ)\to\F.\]

If $\psi_\ell$ is locally constant relative to $T_\ell$ and $x_0$, then
\[\langle\Phi(\hat\psi_\ell)T_\ell,z_0\rangle=m_0\langle\Phi(\Delta_\ell\psi_\ell),z_0\rangle.\]
\end{theorem}

\begin{proof} 
By Lemma~\ref{ad1}, $T_\ell=\sum_{j=1}^J U_{t_j}$.
By Corollary~\ref{C:partial-hecke}, for $1\leq j\leq J$, we have
\[U_{t_j}(z_0)=e_jq(\zeta_j)z_j.\]
Then
\begin{align*}
\langle\Phi(\hat\psi_\ell)T_\ell,z_0\rangle&=\langle\Phi(\hat\psi_\ell),T_\ell z_0\rangle\cr
&=\sum_{j=1}^Jq(\zeta_j)e_j\langle\Phi(\hat\psi_\ell),z_j\rangle\cr
&=\sum_{j=1}^Jq(\zeta_j)e_jm_j'\langle\Phi(\psi_\ell),z_j\rangle.
\end{align*}
We have $e_jm_j'=m_0'd_j$ by Lemma~\ref{H5}, so
\[\langle\Phi(\hat\psi_\ell)T_\ell,z_0\rangle=\sum_{j=1}^Jq(\zeta_j)m_0'd_j\langle\Phi(\psi_\ell),z_j\rangle.\]

Now, for a fixed $j$, we will analyze the term $\langle\Phi(\psi_\ell),z_j\rangle$.  
Recall the definition of the partial Hecke operators $U_{t,j}$ and of the matrices  $s_{\beta,j}$ from the paragraphs before Lemma~\ref{ad1}. Also recall the matrices $t_j$   and the fact 
that $t_j\xS_0=\zeta_j\xS_j$ from the paragraphs before Lemma~\ref{L:stab-reps} .

Because $\Phi$ is invariant under $\KSQ$-homothety, we may choose $a_j,b_j\in K$ so that $\transpose{(a_j,b_j)}\in Y$ represents $\xS_j\in \XSQ$.  In fact, we choose $a_0,b_0\in K$ so that $\transpose{(a_0,b_0)}\in Y$ 
represents $\xS_0$, and then set
\[\column{a_j}{b_j}=\zeta_j^{-1}t_j\column{a_0}{b_0}.\]
We then obtain
\begin{align*}
\langle\Phi(\psi_\ell),z_j\rangle&=\Phi(\psi_\ell)(L(a_j,b_j))\cr
&=\psi_\ell(L(a_j,b_j)_\ell)\prod_{w \in W}\phi_w(L(a_j,b_j)_w).
\end{align*}
Since $t_j$ is an integral matrix with determinant $\ell$, we know that $t_j\in\gl(2,\OK_w)$ for all $w \in W$.  Factor
\[\zeta_j^{-1}=\left(\ell^{f_{\ell,j}}\prod_{w \in W}w^{f_{w,j}}\right)I.\]
(By Definition~\ref{D:q*}, primes not in $W\cup\{\ell\}$ cannot divide the numerators or denominators of the diagonal entries of the matrix $\zeta_j\in Z$.)

Then $L(a_j,b_j)_w$ is the same as the lattice $w^{f_{w,j}}L(a_0,b_0)_w$.  Set
\[c=\prod_{w\in W}\phi_w(L(a_0,b_0)_w).\]
Since  each $\phi_w$ is homogeneous, we obtain
\[\langle\Phi(\psi_\ell),z_j\rangle=\psi_\ell(L(a_j,b_j)_\ell)\ c\prod_{w \in W}q(w^{f_w,j}I).\]
Hence, we see that
\begin{align*}
\langle\Phi(\hat\psi_\ell)T_\ell,z_0\rangle&=\sum_{j=1}^Jq(\zeta_j)d_jm_0'\langle\Phi(\psi_\ell),z_j\rangle\cr
&=cm_0'\sum_{j=1}^Jq(\zeta_j)d_j\psi_\ell(L(a_j,b_j)_\ell)\prod_{w \in W}q(w^{f_{w,j}}I)\cr
&=cm_0'\sum_{j=1}^Jq(\ell^{f_{\ell,j}}I)d_j\psi_\ell(L(a_j,b_j)_\ell),\cr
\end{align*}
where we have used the factorization of $\zeta_j^{-1}$.
 
On the other hand, since $\psi_\ell$ is assumed to be locally constant with respect to $T_\ell$ and $x_0$, and any $s_{\beta,j}$ takes any vertex to a downhill neighbor, we have that for each $s_{\beta,j}$,
\begin{align*}
\psi_\ell(s_{\beta,j}L(a_0,b_0)_\ell)&=\psi_\ell(t_jL(a_0,b_0)\ell)\cr
&=\psi_\ell(\zeta_jL(a_j,b_j)\ell)\cr
&=\psi_\ell((\ell^{f_{\ell,j}})^{-1}L(a_j,b_j)_\ell)\cr
&=q((\ell^{f_{\ell,j}})^{-1}I)\psi_\ell(L(a_j,b_j)_\ell),
\end{align*}
since $\psi_\ell$ is homogeneous.

Hence, using the fact that $d_j=|B_j|$, we have that
\begin{align*}
\langle\Phi(\Delta_\ell^n\psi_\ell),z_0\rangle&=\Phi(\Delta_\ell^n\psi_\ell)(L(a_0,b_0))\cr
&=(\Delta_\ell^n(\psi_\ell)(L(a_0,b_0)_\ell)\prod_{w \in W}\phi_w(L(a_0,b_0)_w)\cr
&=c(\Delta_\ell^n(\psi_\ell)(L(a_0,b_0)_\ell)\cr
&=c\sum_{j=1}^J\sum_{s_{\beta,j}\in B_j}\psi_\ell(s_{\beta,j}L(a_0,b_0)_\ell)\cr
&=c\sum_{j=1}^Jd_jq((\ell^{f_{\ell,j}})^{-1}I)\psi_\ell(L(a_j,b_j)_\ell),
\end{align*}
where we have used Lemma~\ref{ell+1}.
Multiplying both sides of the last equality by $m_0'$ yields the assertion of the theorem, because $q$ has order $2$.
\end{proof}

\begin{corollary}\label{laplace-eigenvector}
For any $\xS_0\in \A_0$, let $z_0\in H_1(\stab{\xS_0},\F)$ be the corresponding homology generator, and let $L_0$ be a lattice corresponding to $\xS_0$.  Assume that $\psi_\ell$ is locally constant relative to $T_\ell$ and $x_0$.  Further, assume that $\Phi(\hat\psi_\ell)$ is $q$-homogeneous and $\KSQ$-invariant, and that $(\Delta_\ell^n\psi_\ell)((L_0)_\ell)=\mu\psi_\ell((L_0)_\ell)$.  Then
\[\langle\Phi(\hat\psi_\ell)T_\ell,z_0\rangle=\mu\langle\Phi(\hat\psi_\ell),z_0\rangle.\]
\end{corollary}

\begin{proof} From Theorem~\ref{Theorem:hecke-transform} and linearity we have
\begin{align*}
\langle\Phi(\hat\psi_\ell)T_\ell,z_0\rangle&=m_0'\langle\Phi(\Delta_\ell\psi_\ell),z_0\rangle\cr
&=m_0'\langle\Phi(\mu\psi_\ell),z_0\rangle\cr
&=\mu\langle\Phi(m_0'\psi_\ell),z_0\rangle\cr
&=\mu\langle\Phi(\hat\psi_\ell)T_\ell,z_0\rangle.\qedhere
\end{align*}
\end{proof}

\begin{corollary}\label{eigenclass} Assume that $\psi_\ell$ is locally constant relative to $T_\ell$ and every $x\in \XSQ$, that 
$\Phi(\hat\psi_\ell)$ is $q$-homogeneous and $\KSQ$-invariant, and that $\Delta_\ell\psi_\ell=\mu\psi_\ell$.  

Then $\Phi(\hat\psi_\ell)$ is an eigenclass for $T_\ell$ with eigenvector $\mu$ and it is an eigenclass for $T_{\ell,\ell}$ with eigenvector $\theta(\ell)$.
\end{corollary}

\begin{proof} First, we show that $\Phi(\hat\psi_\ell)\neq 0$.  By definition,
\[\Phi(\hat\psi_\ell)(L)=\hat\psi_\ell(L_\ell)\prod_{w\in W}\phi_w(L_w).\]
By construction, $\phi_w(\OK_w)=1$ for every $w\in W$ and $\psi_\ell(\Ol)=1$.  Since $m'_{\Ol}=1$, also 
$\hat\psi_\ell(\OK_\ell)=1$.
Therefore, $\Phi(\hat\psi_\ell)(\OK)=1$.

For any $\xS\in\A_0$, write $z_{\xS}\in H^1(\stab{\xS},\F)$ for the homology generator corresponding to $\xS$.  By Corollary~\ref{laplace-eigenvector} and our hypothesis, for each $x\in\A_0$, we have
\[\langle\Phi(\hat\psi_\ell)T_\ell,z_x\rangle=\langle\mu\Phi(\hat\psi_\ell),z_x\rangle.\]
Since $\Phi(\hat\psi_\ell)$ is in the dual space to $H_1(\gl(2,\Z),\MSQ)$, and $\{z_\xS:\xS\in\A_0\}$ spans $H_1(\gl(2,\Z),\MSQ)$, we are finished with $T_\ell$.

As for $T_{\ell,\ell}$, its action is given by the double coset of the central element $\ell I$.  So this is just a single coset, and its action on homology is given by the central character $q$ on the coefficient module $\MSQ$.  Since $q(\ell I)=\theta(\ell)$,
\[\langle\Phi(\hat\psi_\ell)T_{\ell,\ell},z_x\rangle=\langle\Phi(\hat\psi_\ell),T_{\ell,\ell}z_x\rangle=\langle\Phi(\hat\psi_\ell),\theta(\ell)z_x\rangle=\langle\theta(\ell)\Phi(\hat\psi_\ell),z_x\rangle.\]
Hence, $\Phi(\hat\psi_\ell)T_{\ell,\ell}=\theta(\ell)\Phi(\hat\psi_\ell)$.
\end{proof}

\section{Constructing locally constant eigenfunctions}
Recall that $\theta$ is the quadratic Dirichlet character
 associated to the real quadratic field $K/\Q$, and $q$ is the character on $Z$ defined by setting $q(r I)=\theta(r)$ for $r\in \Z\cap K_{S_0}$ and extending multiplicatively to $Z$.  Since $K$ is real quadratic, $q(-I)=\theta(-1)=1$.  Fix an $\F$-valued character  $\chi$ on the group of ideals of $K$ relatively prime to $N$ for some positive integer $N$.

In this section, we will construct locally constant $q$-homogeneous functions $\psi_\ell^0$ on $\T_\ell^{n_\ell}$ that are  eigenfunctions of the Laplace operator with eigenvalues related to $\chi$. We do this first for inert primes $\ell$.

\begin{theorem}\label{T:inert-psi} Let $\ell$ be a prime of $\Q$ that is inert in $K/\Q$ and does not equal the characteristic of $\F$.  Then there is a locally constant $q$-homogeneous function $\psi_\ell^0\in F(\T_\ell^2)$ 
 that is an eigenvalue of the Laplace operator with eigenvalue 0 and satisfies $\psi_\ell^0(\Ol)=1$.
\end{theorem}

\begin{proof} We define $\psi_\ell^0$ inductively. 

For vertices of tier 0, we define $\psi_\ell^0(\Ol)=1$ and $\psi_\ell^0(\ell\Ol)=\theta(\ell)=-1$.  We see easily that $\psi_\ell^0$ is homogeneous on the vertices of tier $0$.  

On vertices $t^2\in\T_\ell^2$ of tier 1, we define $\psi_\ell^0(t^2)=0$.  Clearly $\psi_\ell^0$ is $\theta$-homogeneous on vertices of tier 1.  In addition, since all downhill neighbors of a vertex of tier 0 have tier 1, we can now compute $\Delta_\ell^2(\psi_\ell^0)$ on vertices of tier 0; we find that its value is $0$, as desired. Finally, $\psi_\ell^0$ is constant on all downhill neighbors of vertices of tier 0.

On each vertex $t^2\in \T_\ell^2$ of tier 2, let $u^2\in \T_\ell^2$ be the unique uphill neighbor of $t^2$ of tier 1, and we let $v^2$ be the unique downhill neighbor of $u^2$ of tier 0.  We define $\psi_\ell^0(t^2)=-\psi_\ell^0(v^2)/\ell$.  Because the unique uphill neighbor of $\ell t^2$ of tier 1 is $\ell u^2$, which has a unique downhill neighbor of tier 0 equal to $\ell v^2$, we see that with this definition, $\psi_\ell^0$ is homogeneous on vertices of tier 1.  In addition, for any vertex $u^2$ of tier 1, $\psi_\ell^0$ is constant on the downhill neighbors of $u^2$ of higher tier, since its value on such vertices depends only on its value on the unique downhill inner neighbor of $u^2$.  Finally, we have constructed $\psi_\ell^0$ so that
\[\Delta_\ell^2(\psi_\ell^0)(u^2)=0\]
for each vertex $u^2$ of tier 1.

We continue; for vertices $t^2\in\T_\ell^2$ of odd tier, we define $\psi_\ell^0(t^2)=0$.  This guarantees that for vertices $u^2$ of even tier, $\Delta_\ell^2(\psi_\ell^0)(u^2)=0$, and that $\psi_\ell^0$ is constant on all downhill neighbors of $u^2$ of higher tier.  Further, with this definition, $\psi_\ell^0(\ell t^2)=0=\theta(\ell)\psi_\ell^0(t^2)$
so that $\psi_\ell^0$ is homogeneous on vertices of odd tier.

For a vertex $t^2\in\T_\ell^2$ of positive even tier, let $u^2$ be the unique uphill inner neighbor of $t^2$, and let $v^2$ be the unique downhill inner neighbor of $u^2$.  We define $\psi_\ell^0(t^2)=-\psi_\ell^0(v^2)/\ell$.  Clearly $\psi_\ell^0$ is constant on all downhill outer neighbors of $u^2$ (since its value on such neighbors depends only on its value on $v^2$).  As in the case of tier 2, we see that $\psi_\ell^0(\ell t^2)=\theta(\ell) \psi_\ell^0(t^2)$, and $\Delta_\ell^2(\psi_\ell^0)(u^2)=\psi_\ell^0(v^2)+\ell(-\psi_\ell^0(v^2)/\ell)=0$.

With this construction, we see that $\psi_\ell^0$ is homogeneous, locally constant, and is an eigenfunction of $\Delta_\ell^2$ with eigenvalue 0.
\end{proof}

\begin{lemma}\label{L:inert-invariant} 
For an inert prime $\ell$, the function $\psi_\ell^0$ defined above is $\gl(2,\Z_\ell)$-invariant.
\end{lemma}

\begin{proof} The action of $\gl(2,\Z_\ell)$ fixes vertices of tier 0, and preserves uphill and downhill neighbors, and the tier of each vertex (Lemma~\ref{L:gl2zl-invariant}). Since these relationships determine the values of $\psi_\ell^0$, the function is $\gl(2,\Z_\ell)$-invariant. 
\end{proof}

For a prime $\ell$ that splits in $K/\Q$ and does not divide $N$, we now prepare to construct a locally constant homogeneous function $\psi_\ell^0\in F(\T_\ell^1)$ that is an eigenfunction of $\Delta_\ell=\Delta_\ell^1$.   For the remainder of this section, we will assume that $\ell$ splits in $K$, that $(\ell)=\lambda\lambda'$ and that $\ell\nmid N$, so that $\chi(\lambda)$ and $\chi(\lambda')$ are defined.  In this case, the function that we construct will depend not only on the real quadratic field $K/\Q$, but also on the character $\chi$. Since we work in $\T_\ell^1=\T_\ell$, the concepts of uphill and downhill neighbors coincide.

We begin by defining some terminology and notation for subsets of $\T_\ell$.

\begin{definition}\label{base} We take $\Ol$ as the basepoint of $\T_\ell$ and denote it by $t_0$.  A {\em descendant} of a vertex $t\in\T_\ell$ is a vertex $t_1\neq t$ such that the path from $t_0$ to $t_1$ passes through $t$.  Denote by $C(t)$ the set of all descendants $t'$ of $t$ such that every vertex of the path from $t$ to $t'$ except possibly $t$ is non-idealistic, and let $\overline C(t)=C(t)\cup\{t\}$.  We call $C(t)$ the {\em open cohort} of $t$, and $\overline C(t)$ the {\em closed cohort} of $t$.
\end{definition}

\begin{definition} A {\em simple chain} starting at a vertex $t\in\T_\ell$ is a collection $C$ consisting of $t$ and descendants of $t$ such that for any pair $t',t''\in C$, one of $t',t''$ is a descendant of the other.  An {\em apartment} in $\T_\ell$ is a union of two infinite simple chains starting at a vertex $t$ and having no other vertices in common.
\end{definition} 

\begin{lemma} Let $t$ be an idealistic point in $\T_\ell$.  
\begin{enumerate}[(1)]
\item If $\ell$ is inert, then $t=t_0$.
\item If $(\ell)=\lambda\lambda'$ splits and $t$ is a distance $k>0$ from $t_0$, then $t=\lambda_\ell^k$ or $t=\lambda_\ell'^k$, and both of these points are a distance $k$ from $t_0$.
\item If $(\ell)$ splits and $k>0$, then $\lambda_\ell^k$ and $\lambda_\ell'^k$ define distinct points in $\T_\ell$.
\item No descendant of an non-idealistic point in $\T_\ell$ is idealistic.
\item The vertices of $\T_\ell$ are partitioned into the closed cohorts $\overline C(t_I)$ as $t_I=I_\ell$ runs over the idealistic points of $\T_\ell$ (where $I$ is an ideal of $\OK$ of $\ell$-power norm.)
\item In the split case, the set of idealistic points of $\T_\ell$ form an apartment, namely 
$$\{{\lambda_\ell^k}|k>0\}\cup\{t_0\}\cup \{{\lambda_\ell'^k}|k>0\}.$$
\end{enumerate}
\end{lemma}

\begin{proof} 
In the discussion following Definition~\ref{D:idealistic}, we proved that the set of idealistic nodes of $\T_\ell$ is
$\{t_0\}$ if $\ell$ is inert and $\{{\lambda_\ell^k}|k>0\}\cup\{t_0\}\cup \{{\lambda_\ell'^k}|k>0\}$ if $\ell$ is split.  Since
$\lambda^k$ has index $\ell$ in $\lambda^{k-1}$, and similarly for the powers of $\lambda'$, 
(1) and (2) are now clear.
As for (3), if $\lambda_\ell^k=\lambda_\ell'^k$, then $\lambda^k=\ell^m\lambda'^k$ for some integer $m$, which is absurd.

 If $\ell$ is inert, (4) and (5) are obvious.

Assume then that $\ell$ splits.  Then the idealistic point $\lambda_\ell^k$ is at the end of a path containing the nodes $t_0,\lambda_\ell,\ldots,\lambda_\ell^k$.  A similar statement holds for $\lambda_\ell'^k$.  Since every non-idealistic node is a descendant of $t_0$ and $\T_\ell$ is a tree, no idealistic point can be a descendant of a non-idealistic point.  Hence (4) holds.

For any node $u\in\T_\ell$ consider the path from $t_0$ to $u$ (possibly of length 0.)  Let $t_I$ be the last idealistic point in this path.  Then $u\in\overline C(t_I)$ is in the closed cohort of this idealistic point.  If $u$ were in the closed cohort of two distinct idealistic points, there would be a nontrivial loop in $\T_\ell$.    Hence, (5) holds.

Finally, (6) is clear, since the set of nonnegative powers of $\lambda$ and of $\lambda'$ each form a simple chain starting at $t_0$.
\end{proof}

\begin{definition} Let $N$ be a positive integer and $c$ be an $\F$-valued multiplicative function on the group  $I_{K}(N)$ of nonzero fractional ideals of $K$ relatively prime to $N$. Fix a prime $\ell$ that does not divide $N$.  Assume that $c$ is trivial on the principal fractional ideal $\ell\OK$.  Define $\hat c\in F(\T_\ell)$ by 
\[\hat c(t)=\begin{cases}0&\text{if $t$ is non-idealistic,}\cr
c(I)&\text{if $t=I_\ell$, where $I$ is an ideal of $\ell$-power index in $\OK$.}\end{cases}
\]
\end{definition}

 \begin{lemma} The function $\hat c$ is well defined.\end{lemma}
\begin{proof}
Suppose $I$ and $J$ are both ideals of $\OK$ of $\ell$-power index, and that 
$I_\ell$ and $J_\ell$ are homothetic in $K_\ell$ by a power of $\ell$.  

If $\ell$ is inert, then $I$ and $J$ are both powers of $\ell$.  They are thus both principal, and we see that $c(I)=c(J)=c(\OK)$.

If $(\ell)=\lambda\lambda'$ splits, then $I=\ell^q\mu^a$ and $J=\ell^r\nu^b$ for nonnegative integers $a,b,q,r$, and $\mu,\nu\in\{\lambda,\lambda'\}$.  The fact that $I_\ell$ and $J_\ell$ are homothetic implies that $\mu=\nu$ and $a=b$, so $I$ and $J$ differ by a factor of $\ell^{q-r}$.  Since $c$ is trivial on $\ell\OK$, $c(I)=c(J)$.
\end{proof}

Let $t\in\T_\ell$.  For any $x$ in the open cohort $C(t)$ of $t$, all of the neighbors of $x$ are in the closed cohort $\overline C(t)$.  Hence, the Laplace operator $\Delta_\ell$ defines a linear map from functions on $\overline C(t)$ to functions on $C(t)$.

\begin{lemma}\label{con split}
Assume that $\ell$ is not equal to the characteristic of $\F$.  Let $\mu\in \F$, and let $t$ be an idealistic point of $\T_\ell$ with closed cohort $\overline C(t)$.  Then there is a unique $\F$-valued function $\theta_{t,\mu}$ on $\overline C(t)$ with the following properties:
\begin{enumerate}[(i)]
\item\label{t-mu1} $\theta_{t,\mu}(t)=1$,
\item\label{t-mu2} $\theta_{t,\mu}(s)=0$ for every $s\in C(t)$ that is distance $1$ from $t$,
\item\label{t-mu4} $\theta_{t,\mu}(s)$ depends only on $\ell$, $\mu$, and the distance from $s$ to $t$,
\item\label{t-mu3} $\Delta_\ell(\theta_{t,\mu})(s)=\mu\theta_{t,\mu}(s)$ for every $s\in C(t)$.
\end{enumerate}
\end{lemma}

\begin{proof}
Define a sequence $a_k\in\F$ for $k\geq 0$ by the recurrence relation $a_0=1$, $a_1=0$, and for $k\geq 2$, 
$$a_k=\frac{\mu a_{k-1}-a_{k-2}}{\ell}.$$
This clearly defines a unique sequence.  For $s$ a distance $k$ from $t$ in $\overline C(t)$, set $\theta_{t,\mu}(s)=a_k$.  With this definition, $\theta_{t,\mu}$ satisfies conditions \ref{t-mu1}, \ref{t-mu2}, and \ref{t-mu4}.

Given a point $s\in C(t)$ a distance $k$ from $t$, $s$ has one neighbor a distance $k-1$ from $t$, and $\ell$ neighbors a distance $k+1$ from $t$.  Hence
\begin{align*}
\Delta_\ell(\theta_{t,\mu})(s)&=a_{k-1}+\ell a_{k+1}\cr
&=a_{k-1}+\ell\left(\frac{\mu a_k-a_{k-1}}\ell\right)\cr
&=\mu a_k\cr
&=\mu\theta_{t,\mu}(s),
\end{align*}   
so $\theta_{t,\mu}$ satisfies condition \ref{t-mu3}.

Conversely, if $\theta_{t,\mu}$ is a function on $\overline C(t)$ satisfying condition  \ref{t-mu4}, then for any $s$ a distance $k$ from $t$, we may define $a_k=\theta_{t,\mu}(s)$. If in addition $\theta_{t,\mu}$ satisfies
 conditions  \ref{t-mu1}, \ref{t-mu2}, \ref{t-mu3}, the $a_k$ satisfy the recurrence relation given above. The uniqueness of $\theta_{t,\mu}$ follows from the uniqueness of the sequence $\{a_k\}$.
\end{proof}

\begin{definition}\label{D:split-psi} Let $\mu\in \F$, and assume $\ell$ does not divide $N$ and does not equal the characteristic of $\F$.  We define $\psi_\ell^0\in F(\T_\ell)$ by
$$\psi_\ell^0(s)=\hat\chi(t)\theta_{t,\mu}(s),$$
where $t\in \T_\ell$ is the unique idealistic vertex with $s\in \overline C(t)$.
\end{definition}

\begin{lemma}\label{lemsplit}
Let $\mu\in F$ and assume that $\ell$ does not divide $N$ and does not equal the characteristic of $\F$.
\begin{enumerate}[(1)]
\item $\psi_\ell^0(\Ol)=1$.
\item  $\psi_\ell^0$ is locally constant with respect to $T_\ell$ and any $x\in \XSQ$.
\item Let $\mu=\chi(\lambda)+\chi(\lambda')$.  Then
$$\Delta_\ell\psi_\ell^0=\mu\psi_\ell^0.$$
\end{enumerate}
\end{lemma}
\begin{proof}
The first assertion is immediate from the definitions.

Let $s$ be any vertex in $\T_\ell$.  We wish to show that $\psi_\ell^0$ is constant on all non-idealistic outer downhill neighbors $u$ of $s$.  Then, by Lemma~\ref{loc-const}, part (2) will hold.

Let $s\in\bar C(t)$.  Then any such $u$ will be in $C(t)$.  Since $\hat\chi(t)$ is constant for all points in $C(t)$, we need only show that $\theta_{t,\mu}(u)$ is constant for all such $u$.  Letting the distance from $t$ to $s$ be $k-1$, the distance from $t$ to $u$ will be $k$.  Hence, the desired constancy follows from Lemma~\ref{con split}\ref{t-mu4}.

Now suppose that $s=t=I_\ell$ is idealistic. Then $s$ has exactly two idealistic neighbors, namely $(\lambda I)_\ell$ and $(\lambda' I)_\ell$.  The nonidealistic neighbors $u$ of $s$ are all in $C(t)$ and have distance 1 from $t$; hence $\theta_{t,\mu}$ vanishes on them all.  Hence
$$(\Delta_\ell\psi_\ell^0)(s)=\chi(\lambda I)+\chi(\lambda'I)=(\chi(\lambda)+\chi(\lambda'))\chi(I)=\mu\psi_\ell^0(t).$$

Finally, suppose that $s$ is non-idealistic and belongs to the open cohort $C(t)$.  Then
\begin{align*}
(\Delta_\ell\psi_\ell^0)(s)&=\sum_u\psi_\ell^0(u)\cr
&=\sum_u\hat\chi(t)\theta_{t,\mu}(u)\cr
&=\hat\chi(t)\sum_u\theta_{t,\mu}(u)\cr
&=\hat\chi(t)(\Delta_\ell\theta_{t,\mu})(s)\cr
&=\mu\hat\chi(t)\theta_{t,\mu}(s)\cr
&=\mu\psi_\ell^0(s),
\end{align*}
by Lemma~\ref{con split}\ref{t-mu3}, where the sums run over all neighbors $u$ of $s$.
\end{proof}

\section{$\KSQ$-invariance}\label{S:Invariance}

\begin{lemma}\label{aL}
Fix a prime $\ell$ that is unramified in $K$, and let $n=1$ if $\ell$ splits in $K$ and $2$ if $\ell$ is inert.  Let $L$ be a $\Z$-lattice in $K$, and let $\alpha\in K^\times$.  Let $s^n$ be the vertex in $\T_\ell^n$ corresponding to $L_\ell$, and let $u^n$ be the vertex corresponding to $(\alpha L)_\ell$.  Factor the fractional ideal $\alpha\OK=I_1I_2$, where $I_1$ has norm a power of $\ell$ and $I_2$ is prime to $\ell$.
\begin{enumerate}[(1)]
\item There exists a matrix $g\in \GL(2,\Q_\ell)$ depending only on $\alpha$ (independent of $L$), such that $u^n=gs^n$.  If $\ell$ is inert, then $g=\ell^k g'$ with $k\in\Z$ and $g'\in\gl(2,\Z_\ell)$.
\item The vertex $s^n$ is idealistic if and only if $u^n$ is idealistic.  If $s^n$ corresponds to $L_\ell$ with $L$ an ideal, then $u^n$ corresponds $(I_1L)_\ell$.
\item Suppose $\ell$ is split.  Assume that $s^n$ is not idealistic, but lies in the open cohort $C(t)$ of the idealistic point $t^n=M_\ell$, where $M$ is an ideal of $\ell$-power norm.  Then $u^n$ lies in the open cohort $C(t_1^n)$, where $t_1^n=(I_1M)_\ell$ and the distance between $s^n$ and $t^n$ is the same as the distance between $u^n$ and $t_1^n$.
\end{enumerate}
\end{lemma}
\begin{proof} 
(1) First, suppose that $\ell$ is inert.  Via our identification of $K_\ell$ with $\Q_\ell^2$, multiplication by $\alpha$ is a 
$\Q_\ell$-linear isomorphism from $\Q_\ell^2$ to $\Q_\ell^2$; hence, it is given by a matrix $g\in\GL(2,\Q_\ell)$.  We can write $\alpha\in K_\ell$ as $\alpha=\ell^k\eta$ for some $k\in\Z$, and some unit $\eta\in\Ol^\times$; multiplication by $\eta$ is given by a matrix in $\gl(2,\Z_\ell)$.

Now assume that $\ell$ is split.  In this case, we identify $K_\ell$ with $\Q_\ell^2$ by mapping $\alpha$ to $(\alpha,\alpha')$.  Then multiplication by $\alpha$ is defined by the matrix
$$\begin{pmatrix}\alpha&0\cr0&\alpha'\end{pmatrix},$$
which is in $\GL(2,\Q_\ell)$.

(2) $L$ is a fractional ideal if and only if $\alpha L$ is a fractional ideal.  If $L=MP$ with $M$ a fractional ideal of $\ell$-power norm, and $P$ a fractional ideal prime to $\ell$, then
$$(aL)_\ell=(I_1M)_\ell=(I_1L)_\ell.$$

(3) Let $g\in\GL(2,\Q_\ell)$ be the matrix from part (1) corresponding to multiplication by $\alpha$.  Multiplication by $g$ is then an isometry of $\T_\ell$ that takes idealistic vertices to idealistic vertices, and non-idealistic vertices to non-idealistic vertices. Let $R$ be a simple path from $t^n$ to $s^n$ whose only idealistic vertex is $t^n$.  Then $gR$ is a simple path from $gt^n$ to $u^n$ of the same length as $R$, whose only idealistic vertex is $gt^n$.  Moreover, $u^n$ lies in the open cohort $C(gt^n)$ where $gt^n=(I_1M)_\ell$.\qedhere
\end{proof}

\begin{theorem}\label{L:KSQ-invariant} Let $\F$ be a field of characteristic $0$ or of finite characteristic not equal to two.  If $\F$ has characteristic 0, set $p=1$, and otherwise let $p$ be the characteristic of $\F$.
Assume that $\chi$ is trivial on principal ideals generated by elements of $\Q^\times\cap K_{S_0}$.  Also assume that $\chi$ is trivial on principal ideals generated by elements of $\KSQ$.
Let $\Phi$ be the function from lattices in $K$ to $\F$ defined by
\[\Phi(L)=\prod_{w\nmid pdN}\hat\psi_w^0(L_w).\]
Then $\Phi(\alpha L)=\Phi(L)$ for all $\alpha\in\KSQ$ and all lattices $L$ in $K$.

Moreover, $\Phi(\alpha L)=q(\alpha I)\Phi(L)$ for all $\alpha\in \Q^\times\cap K_{S_0}$.
\end{theorem}

\begin{proof}
Let $L=L(c,d)$ be a lattice in $K$ and let $\alpha\in\KSQ$.  Note that $m_L'=m_{\alpha L}'$, since $K^\times$ is commutative.  Hence, there is a single integer $m'$, such that for each prime $w\nmid pdN$, we have
\[\hat\psi_w^0(L_w)=m'\psi_w^0(L_w)\]
and
\[\hat\psi_w^0((\alpha L)_w)=m'\psi_w^0((\alpha L)_w).\]

Assume first that $w$ is inert in $K$.  Then we may factor $\alpha\OK$ as 
\[\alpha\OK=w^jI_2,\]
with $I_2$ a fractional ideal that is relatively prime to $w$.  By Lemma~\ref{aL}(1), we have
\[(\alpha L)_w=w^jgL_w\]
for some $g\in\gl(2,\Z_w)$.  Since $\psi_w^0$ is homogeneous and is $\gl(2,\Z_w)$-invariant on $\T_\ell^2$ (by Lemma~\ref{L:inert-invariant}),
 we have
\[\hat\psi_w^0((\alpha L)_w)=m'\psi_w^0(w^jgL_w)=m'q^*(w^j)\psi_w^0(L_w)=q^*(w^j)\hat\psi_w^0(L_w).\]

Now assume that $w$ splits in $K$.  Let $s\in\T_\ell$ be the vertex corresponding to $L_w$, and let $u$ correspond to $(\alpha L)_w$.

If $s$ is idealistic, so is $u$, and we see that
\[\psi_w^0(s)=\hat\chi(s)\theta_{s,\mu}(s)=\hat\chi(s)=\chi(L)=\chi(\alpha L)=\hat\chi(u)=
\hat\chi(u)\theta_{u,\mu}(u)\psi_w^0(u).\]

If $s$ is nonidealistic, then so is $u$, and $u=gs$ for some $g\in\gl(2,\Q_w)$.  Suppose $s$ lies in the open cohort $C(t)$ of the idealistic vertex $t$ corresponding to $I_w$, where $I$ is an ideal of $w$-power index in $\OK$.  By Lemma~\ref{aL}(3), $u$ is in the open cohort $C(t_1)$ of the idealistic point $t_1$ corresponding to $(I_1I)_w$, 
where $\alpha\OK=I_1I_2$, with $I_1$ having norm a power of $w$, and $I_2$ having norm relatively prime to $w$.  In addition, the distance from $s$ to $t$ is the same as the distance from $u$ to $t_1$.  Hence,
\[\hat\chi(t_1)=\chi(I_1I)=\chi(I_1)\chi(I)=\chi(I_1)\hat\chi(t)\]
and
\[\theta_{t,\mu}(s)=\theta_{t_1,\mu}(u).\]
Therefore,
\begin{align*}
\hat\psi_w^0((\alpha L)_w)&=m'\psi_w^0(u)\cr
&=m'\hat\chi(t_1)\theta_{t_1,\mu}(u)\cr
&=m'\chi(I_1I)\theta_{t_1,\mu}(u)\cr
&=m'\chi(I_1)\hat\chi(t)\theta_{t,\mu}(s)\cr
&=\chi(I_1)\hat\psi_w^0(L_w).
\end{align*}

In all of this, the fractional ideal $I_1$ depends on $w$; we will call it $\wpart$.
Then $\wpart$ is a product of powers of primes lying over $w$; if $w$ is inert, it is clear that $\wpart$ is principal with a generator $\beta_\alpha(w)$ in $\Q^\times\cap K_{S_0}$, so that $\chi(\wpart)=1$.  

Since $\alpha\in\KSQ$,  $\alpha$ is relatively prime to $pdN$, so that
\[\alpha\OK=\prod_{w\nmid pdN}\wpart=\left(\prod_{w\text{ inert}}\wpart\right)\left(\prod_{w\text{ split}}\wpart\right).\]
Setting $\beta=\prod_{w\text{ inert}}\beta_\alpha(w)$, we have
\[\beta\OK=\left(\prod_{w\text{ inert}}\wpart\right).\]
Since $\alpha\in\KSQ$, $q^*(\alpha)=1$.  Because $q^*$ depends only on inert prime factors, and the powers of inert primes dividing $\alpha$ and $\beta$ are equal, we see that
\[1=q^*(\alpha)=q^*(\beta).\]  In addition, we have that $\chi(\beta\OK)=1$, since $\beta$ is a product of powers of elements of $\Q^\times\cap K_{S_0}$, and we have assumed that $\chi$ is trivial on ideals generated by elements of $\Q^\times\cap K_{S_0}$.  
  Hence, we see that
\[ \prod_{w\text{ split}}\wpart\]
is principal, with generator $\alpha/\beta$, so
\[\prod_{w\text{ split}}\chi(\wpart)=\frac{\chi(\alpha\OK)}{\chi(\beta\OK)}=\chi(\alpha\OK)=1,\]
since we have assumed that $\chi$ is trivial on principal ideals generated by elements of $\KSQ$.

Hence, we obtain
\begin{align*}
\Phi(\alpha L)&=\prod_{w\nmid pdN}\hat\psi_w^0((\alpha L)_w)\cr
&=\left(\prod_{w\text{ inert}}\hat\psi_w^0((\alpha L)_w)\right)\left( \prod_{w\text{ split}}\hat\psi_w^0((\alpha L)_w)\right)\cr
&=\left(\prod_{w\text{ inert}}q^*(\beta_\alpha(w))\hat\psi_w^0((L)_w)\right)\left( \prod_{w\text{ split}}\chi(\wpart)\hat\psi_w^0((L)_w)\right)\cr
&=q^*(\beta)\left( \prod_{w\text{ split}}\chi(\wpart)\right)\prod_{w\nmid pdN}\hat\psi_w^0((L)_w)
\cr
&=\Phi(L).\qedhere
\end{align*}

Finally, if $\alpha\in \Q^\times\cap K_{S_0}$, it is a product of powers of primes not dividing $pdN$. 
We may thus assume that $\alpha$ is such a prime.  The $q$-homogeneity of $\Phi$ then follows by Lemma~\ref{new} from the homogeneity of the individual $\psi_w^0$ functions (see Theorem~\ref{T:inert-psi} for inert primes, and note that homogeneity is trivial for split primes).
\end{proof}

\section{Galois representations}

We now define the Galois representations to which our main theorem below applies.  

As before, we let $K$ be a real quadratic field of discriminant $d$, cut out by the Dirichlet character $\theta$.  Let $\F$ be a field of characteristic 0 (in which case we set $p=1$) or a field of odd characteristic $p$, let $G_K$ be the absolute Galois group of $K$ (i.e.\  $\gal(\bar\Q/K)$), and let $\chi:G_K\to\F^\times$ be a character of $G_K$ with finite image.  By class field theory, we can think of $\chi$ as a character on the group of the nonzero fractional ideals of $K$ relatively prime to $N$ for some positive $N\in\Z$.  Let $L$ be the fixed field of the kernel of $\chi$.  Then $L/K$ is Galois.  
We fix an $M\geq3$ that divides $pdN$ and define $S_0$ and $S$ as in Definition~\ref{D:S_0}.

We place the following conditions on the character $\chi$.

\begin{enumerate}[(1)]
\item $\chi$ is trivial on the principal fractional ideals of $K$ generated by elements of $\KSQ$.
\item $\chi$ is trivial on the principal fractional ideals of $K$ generated by elements of $\Q^\times\cap K_{S_0}$.
\item $[L:K]$ is odd.
\item $L/\Q$ is Galois.
\end{enumerate}

An example of such a $\chi$ would be any unramified character of $G_K$ of odd order; such a character would be trivial on all principal fractional ideals of $K$, and $L$ would be a subfield of the Hilbert class field of $K$ and hence be Galois over $\Q$.

Let $\rho:G_\Q\to\gl(2,\F)$ be the induced representation 
\[\rho=\ind_{G_K}^{G_\Q}\chi.\]
Note that this representation will factor through $\gal(L/\Q)$.  We have an exact sequence
\[1\to\gal(L/K)\to\gal(L/\Q)\to\gal(K/\Q)\to1;\]
since $[L:K]$ is odd, this sequence splits, so there is an element $\tau$ of order $2$ in $\Gal(L/\Q)$ mapping to the nonidentity element of $\gal(K/\Q)$; we can lift it to an element $\tau\in G_\Q$, and we have that $\tau^2$ is the identity modulo $G_L$.

With respect to a suitable basis, it is easy to see that for $g\in G_\Q$, we have the following:
\begin{enumerate}[(a)]
\item  If $g\in G_K$, then 
\[\rho(g)=\begin{pmatrix}\chi(g)&0\cr0&\chi(g')\end{pmatrix},\]
where $g'=\tau^{-1}g\tau$.
\item If $g\notin G_K$, then $g=h\tau$ for some $h\in G_K$, and 
\[\rho(g)=\begin{pmatrix}0&\chi(h')\cr\chi(h\tau^2)&0\end{pmatrix},\]
where $h'=\tau^{-1}h\tau$.
\end{enumerate}

If we now let $g$ be a Frobenius element in $G_\Q$ for some prime $\ell$ of $\Q$ not dividing $pdN$ (so that $\ell$ is unramified in $L/\Q$), then we have the following two cases.

If $\ell$ splits in $K$ and $\ell\nmid N$, then $g\in G_K$.  If we write $\ell\OK=\lambda\lambda'$ with $\lambda,\lambda'$ primes in $K$, then we may take $g$ to be a Frobenius in $G_K$ of $\lambda$; a Frobenius of $\lambda'$ will be $g'$.  Hence, we have 
$$\tr(\rho(g))=\chi(g)+\chi(g')=\chi(\lambda)+\chi(\lambda'),$$
 and
\[\det(\rho(g))=\chi(g)\chi(g')=\chi(\lambda)\chi(\lambda')=\chi(\lambda\lambda')=\chi(\ell\OK)=1\]
by condition (2) on the character $\chi$.

On the other hand, if $\ell$ is inert in $K$ and $\ell\nmid N$, write $g=h\tau$ as above.
Then
$$\tr(\rho(g))=0$$
 and $\det(\rho(g))=-\chi(h\tau^2)\chi(h')$ with $h'=\tau^{-1}h\tau$.  We note that $g^2$ is a Frobenius of $\ell\OK$ in $G_K$.  Hence, we have
\[\det(\rho(g))=-\chi(h\tau^2)\chi(h')=-\chi(h\tau^2h')=-\chi((h\tau)^2)=-\chi(g^2)=-\chi(\ell\OK)=-1,\]
where we have again used condition (2) on $\chi$.

Note that in each case, when $g$ is a Frobenius in $G_\Q$ of $\ell$, we have $\det(\rho(g))=\theta(\ell)$.

Now we check that $\rho$ is even.  Let $c\in G_\Q$ be a complex conjugation.  Since $c$ has order 2 and $\chi$ has odd order, $\chi(c)=\chi(\tau^{-1}c\tau)=1$.  From the formula in (a), since $c\in G_K$,  $\rho(c)$ is the identity matrix.

\begin{theorem}\label{T:main}
Let $K$ be a real quadratic field of discriminant $d$, let $\F$ be a field of characteristic 0 or a finite field of odd characteristic.  In the first case set $p=1$ and in the second case let $p$ be the characteristic of $\F$.
Let $\chi:G_K\to\F^\times$ be a character with finite image.   Let $L$ be the fixed field of the kernel of $\chi$ and choose $N\in\Z$ so that $L/K$ is unramified outside primes of $K$ dividing $N$. Let $M\geq 3$, $S_0=S_0(pdN)$, $S=S(M)\cap S_0$, $\theta$ the Dirichlet character cutting out $K$,  
$q$ the character of $Z$ determined by $q(rI)=\theta(r)$ for all $r\in \Z\cap K_{S_0}$, 
and  $\MSQ$ the module defined in Definition~\ref{D:MSQ}.  Assume
\begin{enumerate}[(1)]
\item $\chi$ is trivial on the principal fractional ideals of $K$ generated by elements of $\KSQ$.
\item $\chi$ is trivial on the principal fractional ideals of $K$ generated by elements of $\Q^\times\cap K_{S_0}$.
\item $[L:K]$ is odd.
\item $L/\Q$ is Galois.
\end{enumerate}
Then $\rho:G_\Q\to\gl(2,\F)$ given by $\rho=\ind_{G_K}^{G_\Q}\chi$ is an even Galois representation, and is attached to a Hecke eigenclass in $H^1(\gl(2,\Z),\MSQ^*)$.
\end{theorem}

\begin{proof} Given $\chi$ satisfying the conditions of the theorem, we define 
\[\Phi(L)=\prod_{w\nmid pdN}\hat\psi_w^0(L_w)\]
where $\psi_w^0$ is the function constructed in the proof of Theorem~\ref{T:inert-psi} for $w$ inert in $K$ and prime to $pN$, and the function defined by Definition~\ref{D:split-psi} for $w$ splitting in $K$ and prime to $pN$. 

By Theorem~\ref{L:KSQ-invariant}, $\Phi$ is $\KSQ$-invariant and $q$-homogeneous.  Hence, by Lemma~\ref{Lemma:homothety-to-cohomology} we may consider it as an element of $H^1(\gl(2,\Q),\MSQ^*)$.    By Corollary~\ref{eigenclass}, combined with Lemma~\ref{lemsplit} and Theorem~\ref{T:inert-psi} we see that for all $\ell$ unramified in $L/\Q$,  $\Phi$ is an eigenvector for $T_\ell$ and $T_{\ell,\ell}$, and that the eigenvalues of $T_\ell$ match the trace of $\rho(\frob_\ell)$.  The $q$-homogeneity of $\Phi$ shows that the eigenvalues of $T_{\ell,\ell}$ match the determinant of $\rho(\frob_\ell)$ for all $\ell$ unramified in $L/\Q$.  Hence, $\Phi$ is attached to $\rho$.  
\end{proof}


\bibliographystyle{plain}
\bibliography{biblio}

\begin{thebibliography}{10}

\bibitem{AD1}
Avner Ash and Darrin Doud.
\newblock Reducible {G}alois representations and the homology of {${\rm
  GL}(3,\mathbb Z)$}.
\newblock {\em Int. Math. Res. Not. IMRN}, (5):1379--1408, 2014.

\bibitem{ADP}
Avner Ash, Darrin Doud, and David Pollack.
\newblock Galois representations with conjectural connections to arithmetic
  cohomology.
\newblock {\em Duke Math. J.}, 112(3):521--579, 2002.

\bibitem{AS}
Avner Ash and Warren Sinnott.
\newblock An analogue of {S}erre's conjecture for {G}alois representations and
  {H}ecke eigenclasses in the mod {$p$} cohomology of {${\rm GL}(n,{\bf Z})$}.
\newblock {\em Duke Math. J.}, 105(1):1--24, 2000.

\bibitem{Brown}
Kenneth~S. Brown.
\newblock {\em Cohomology of groups}, volume~87 of {\em Graduate Texts in
  Mathematics}.
\newblock Springer-Verlag, New York-Berlin, 1982.

\bibitem{Caraiani}
Ana Caraiani and Bao~V. Le~Hung.
\newblock On the image of complex conjugation in certain {G}alois
  representations.
\newblock {\em Compos. Math.}, 152(7):1476--1488, 2016.

\bibitem{Casselman}
Bill Casselman.
\newblock The {B}ruhat-{T}its tree of {SL(2)}.
\newblock {\em Unpublished}, 2016.
\newblock Available from \newline{\tt
  https://www.math.ubc.ca/$\sim$cass/research/pdf/Tree.pdf}.

\bibitem{Eisenbud}
David Eisenbud.
\newblock {\em Commutative algebra}, volume 150 of {\em Graduate Texts in
  Mathematics}.
\newblock Springer-Verlag, New York, 1995.

\bibitem{HLTT}
Michael Harris, Kai-Wen Lan, Richard Taylor, and Jack Thorne.
\newblock On the rigid cohomology of certain {S}himura varieties.
\newblock {\em Res. Math. Sci.}, 3:Paper No. 37, 308, 2016.

\bibitem{Herzig}
Florian Herzig.
\newblock The weight in a {S}erre-type conjecture for tame {$n$}-dimensional
  {G}alois representations.
\newblock {\em Duke Math. J.}, 149(1):37--116, 2009.

\bibitem{KW1}
Chandrashekhar Khare and Jean-Pierre Wintenberger.
\newblock Serre's modularity conjecture. {I}.
\newblock {\em Invent. Math.}, 178(3):485--504, 2009.

\bibitem{KW2}
Chandrashekhar Khare and Jean-Pierre Wintenberger.
\newblock Serre's modularity conjecture. {II}.
\newblock {\em Invent. Math.}, 178(3):505--586, 2009.

\bibitem{Scholze}
Peter Scholze.
\newblock On torsion in the cohomology of locally symmetric varieties.
\newblock {\em Ann. of Math. (2)}, 182(3):945--1066, 2015.

\bibitem{Serre-1987}
Jean-Pierre Serre.
\newblock Sur les repr\'esentations modulaires de degr\'e {$2$} de {${\rm
  Gal}(\overline{\bf Q}/{\bf Q})$}.
\newblock {\em Duke Math. J.}, 54(1):179--230, 1987.

\bibitem{Serre}
Jean-Pierre Serre.
\newblock {\em Trees}.
\newblock Springer Monographs in Mathematics. Springer-Verlag, Berlin, 2003.
\newblock Translated from the French original by John Stillwell, Corrected 2nd
  printing of the 1980 English translation.

\bibitem{Weil}
Andr{\'e} Weil.
\newblock {\em Basic number theory}.
\newblock Classics in Mathematics. Springer-Verlag, Berlin, 1995.
\newblock Reprint of the second (1973) edition.

\end{thebibliography}

\end{document}